\documentclass[11pt]{amsart}

\usepackage[pdfauthor={Ziyang GAO}, 
        pdftitle={Ax-Schanuel}%
      dvips,colorlinks=true]{hyperref}
\hypersetup{
  bookmarksnumbered=true,
  linkcolor=black,
  citecolor=black,
  pagecolor=black, 
  urlcolor=black,  
}

\usepackage[utf8]{inputenc}

\usepackage{xtab}
\usepackage{amsmath, amssymb, cancel,verbatim}
\usepackage[all]{xy}
\usepackage[scr=euler,scrscaled=.96]{mathalfa}

\usepackage{mathtools,booktabs}
\usepackage{amscd}
\usepackage{bbm, stmaryrd}
\usepackage{yfonts}
\usepackage{color}

\usepackage{epstopdf} 
\usepackage{booktabs}

\newtheorem{teo}{Theorem}[section]
\newtheorem{thm}[teo]{Theorem}
\newtheorem{prop}[teo]{Proposition}
\newtheorem{lemma}[teo]{Lemma}
\newtheorem{cor}[teo]{Corollary}

\newtheorem{eg}[teo]{Example}
\newtheorem{defn}[teo]{Definition}
\newtheorem{ques}[teo]{Question}
\newtheorem{rmk}[teo]{Remark}
\newtheorem{fact}{Fact}

\newtheorem{def-prop}[teo]{Definition-Proposition}
\newtheorem{notation}[teo]{Notation}

\numberwithin{equation}{section}

  \newcommand{\A}{\mathbb{A}}
  \newcommand{\C}{\mathbb{C}}

  \newcommand{\G}{\mathbb{G}}

  \newcommand{\N}{\mathbb{N}}
  
  \newcommand{\Q}{\mathbb{Q}}
  \newcommand{\R}{\mathbb{R}}  
  \renewcommand{\S}{\mathbb{S}}
  
  \newcommand{\Z}{\mathbb{Z}}

  \newcommand{\bu}{\mathbf{u}}
  \newcommand{\bH}{\mathbf{H}}
  
  \renewcommand{\epsilon}{\varepsilon}

  \newcommand{\im}{\text{Im}}

  \renewcommand{\cong}{\simeq}
  \renewcommand{\bar}{\overline}
  \renewcommand{\tilde}{\widetilde}

  \providecommand{\frac}[1]{\operatorname{Frac}(#1)}

  \renewcommand{\hom}{\operatorname{Hom}}

  \newcommand{\stab}{\operatorname{Stab}}
  
  \newcommand{\GL}{\operatorname{GL}}

  \renewcommand{\lim}{\operatorname{lim}}
  \renewcommand{\deg}{\operatorname{deg}}

  \newcommand{\lie}{\operatorname{Lie}}

  \newcommand{\der}{\mathrm{der}}

  \newcommand{\Zar}{\mathrm{Zar}}
  \newcommand{\biZar}{\mathrm{biZar}}
  \newcommand{\GSp}{\mathrm{GSp}}
  \newcommand{\Sp}{\mathrm{Sp}}

\newcommand{\cA}{\mathcal{A}}

\newcommand{\cD}{\mathcal{D}}

\newcommand{\cG}{\mathcal{G}}

\newcommand{\cR}{\mathcal{R}}

\newcommand{\cT}{\mathcal{T}}

\newcommand{\cX}{\mathcal{X}}
\newcommand{\cY}{\mathcal{Y}}

\newcommand{\pullbackcorner}[1][dr]{\save*!/#1-1.7pc/#1:(-1.5,1.5)@^{|-}\restore}

\newcommand\supervisor[1]{\def\@supervisor{#1}}

\newcounter{elno}

\renewcommand{\cong}{\simeq}

\begin{document}
\title[Ax-Schanuel for the universal abelian variety]{Mixed Ax-Schanuel for the universal abelian varieties and some applications}
\author{Ziyang Gao}

\address{CNRS, IMJ-PRG, 4 place de Jussieu, 75005 Paris, France;  Department of Mathematics, Princeton University, Princeton, NJ 08544, USA}
\email{ziyang.gao@imj-prg.fr}

\subjclass[2010]{11G10, 11G18, 14G35}

\maketitle

\begin{abstract} In this paper we prove the mixed Ax-Schanuel theorem for the universal abelian varieties (more generally any mixed Shimura variety of Kuga type), and give some simple applications. In particular we present an application to studying the generic rank of the Betti map.
\end{abstract}

\tableofcontents

\section{Introduction}
The goal of this paper is to prove a transcendence result and give some simple applications. More applications to Diophantine problems will be given in forthcoming papers.

The transcendence result is the following mixed Ax-Schanuel theorem for the universal abelian variety. We call it ``mixed'' since the ambient space is a mixed Shimura variety but not a \textit{pure} Shimura variety. It parametrizes $1$-motives of a certain kind. The result is an extension of a recent result of Mok-Pila-Tsimerman \cite{MokAx-Schanuel-for} on the Ax-Schanuel theorem for pure Shimura varieties. Let us describe the setting.

Let $D = \mathrm{diag}(d_1, \cdots, d_g)$ be a diagonal $g \times g$-matrix with $d_1|\cdots|d_g$ positive integers. Let $\A_{g,D}(N)$ be the moduli space of abelian varieties of dimension $g$ which are polarized of type $D$ equipped with level-$N$-structures. Assume $N \ge 3$. Then $\A_{g,D}(N)$ admits a universal family $\pi \colon \mathfrak{A}_{g,D}(N) \rightarrow \A_{g,D}(N)$. For simplicity we drop the ``$D$'' and the ``$(N)$''.

The uniformizing space $\cX_{2g,\mathrm{a}}^+$ of $\mathfrak{A}_g$, in the category of complex spaces, admits a reasonable algebraic structure. This is explained in $\mathsection$\ref{SubsectionReviewOnBiAlgSystem}. Denote by $\bu \colon \cX_{2g,\mathrm{a}}^+ \rightarrow \mathfrak{A}_g$ the uniformization. We say that an irreducible subvariety $Y$ of $\mathfrak{A}_g$ is \textit{bi-algebraic} if one (and hence any) complex analytic irreducible component of $\bu^{-1}(Y)$ is algebraic in $\cX_{2g,\mathrm{a}}^+$. There is a good geometric interpretation of bi-algebraic subvarieties of $\mathfrak{A}_g$; see $\mathsection$\ref{SubsectionGeomDescriptionBiAlg}.

The main result of the paper is the following theorem, which we prove as Theorem~\ref{ThmASUnivAbVar}. Denote by $\mathrm{pr}_{\cX_{2g,\mathrm{a}}^+} \colon \cX_{2g,\mathrm{a}}^+ \times \mathfrak{A}_g \rightarrow \cX_{2g,\mathrm{a}}^+$ and $\mathrm{pr}_{\mathfrak{A}_g} \colon \cX_{2g,\mathrm{a}}^+ \times \mathfrak{A}_g \rightarrow \mathfrak{A}_g$ the natural projections.
\begin{thm}[mixed Ax-Schanuel for the universal abelian variety]\label{ThmASUnivAbVarIntro}
Let $\mathscr{Z}$ be a complex analytic irreducible subvariety of $\mathrm{graph}(\bu) \subseteq \cX_{2g,\mathrm{a}}^+ \times \mathfrak{A}_g$. Denote by $Z = \mathrm{pr}_{\mathfrak{A}_g}(\mathscr{Z})$. Then
\[
\dim \mathscr{Z}^{\Zar} - \dim \mathscr{Z} \ge \dim Z^{\biZar},
\]
where $\mathscr{Z}^{\Zar}$ is the Zariski closure of $\mathscr{Z}$ in $\cX_{2g,\mathrm{a}}^+ \times \mathfrak{A}_g$, and $Z^{\biZar}$ is the smallest bi-algebraic subvariety of $\mathfrak{A}_g$ which contains $Z$.

Moreover if we denote by $\tilde{Z} = \mathrm{pr}_{\cX_{2g,\mathrm{a}}^+}(\mathscr{Z})$, then the equality holds if $\tilde{Z}$ is a complex analytic irreducible component of $\tilde{Z}^{\Zar} \cap \bu^{-1}(Z^{\Zar})$.
\end{thm}

The theorem is motivated by Schanuel's Conjecture on transcendental number theory. The analogue of this conjecture over function fields, currently known as the \textit{Ax-Schanuel theorem for complex algebraic tori}, was proven by Ax \cite{AxOn-Schanuels-co}. Later on Ax generalized his result to complex semi-abelian varieties \cite{AxSome-topics-in-}. The result of Ax \cite{AxOn-Schanuels-co} was reformulated and re-proven by Tsimerman \cite{TsimermanAx-Schanuel-and} using o-minimal geometry. In the Shimura setting, the theorem was proven by Pila-Tsimerman for $Y(1)^N$ \cite{PilaAx-Schanuel-for}, and recently proven for any pure Shimura variety by Mok-Pila-Tsimerman \cite{MokAx-Schanuel-for}. Its generalization to any variation of pure Hodge structures is proven by Bakker-Tsimerman \cite{BakkerThe-Ax-Schanuel}. A particular case of the Ax-Schanuel theorem, called the \textit{Ax-Lindemann} theorem, concerns the case where $\tilde{Z}$ is assumed to be algebraic. We refer to the survey \cite{KUYProceedingUtah} for its history.

Our proof of Theorem~\ref{ThmASUnivAbVarIntro} uses the work of Mok-Pila-Tsimerman and extends their proof. As a statement itself, Theorem~\ref{ThmASUnivAbVarIntro} implies the pure Ax-Schanuel theorem for the moduli space $\A_g$.

\subsection*{Application to the Betti map}
Let $\cA \rightarrow S$ be an abelian scheme of relative dimension $g$ over a smooth irreducible complex algebraic variety. 
Let $\tilde{S}\rightarrow S^{\mathrm{an}}$ be the universal covering, and let $\cA_{\tilde{S}}:= \cA \times_S \tilde{S}$. 
Then we can define the \textit{Betti map} (which is real analytic)
\[
b \colon \cA_{\tilde{S}} \rightarrow \mathbb{T}^{2g}
\]
where $\mathbb{T}^{2g}$ is the real torus of dimension $2g$ as follows. As $\tilde{S}$ is simply-connected, one can define a basis $\omega_1(\tilde{s}),\ldots,\omega_{2g}(\tilde{s})$ of the period lattice of each fiber $\tilde{s} \in \tilde{S}$ as holomorphic functions of $s$. Now each fiber $\cA_{\tilde{s}}$ can be identified with the complex torus $\C^g/\Z \omega_1(\tilde{s})\oplus \cdots \oplus \Z\omega_{2g}(\tilde{s})$, and each point $x \in \cA_{\tilde{s}}(\C)$ can be expressed as the class of $\sum_{i=1}^{2g}b_i(x) \omega_i(\tilde{s})$ for real numbers $b_1(x),\ldots,b_{2g}(x)$. Then $b(x)$ is defined to be the class of the $2g$-tuple $(b_1(x),\ldots,b_{2g}(x)) \in \R^{2g}$ modulo $\Z^{2g}$.

Let $\xi$ be a multi-section of $\cA/S$. Then it induces a multi-section $\tilde{\xi}$ of $\cA_{\tilde{S}}/\tilde{S}$. The following question is asked by Andr\'{e}-Corvaja-Zannier \cite[Question~2.1.2]{ACZBetti}.
\begin{ques}\label{ConjBettiRank}
Assume $\Z\xi$ is Zariski dense in $\cA$ and that $\cA/S$ has no fixed part (over any finite \'{e}tale covering of $S$). If $\dim \iota(\xi(S)) \ge g$, is it true that
\begin{equation}\label{EqQuesACZ}
\max_{\tilde{s} \in \tilde{S}}\left(\mathrm{rank}(\mathrm{d}b|_{\tilde{\xi}(\tilde{s})}) \right) = 2g?
\end{equation}
\end{ques}
Andr\'{e}-Corvaja-Zannier systematically studied Question~\ref{ConjBettiRank}. One important idea of Andr\'{e}-Corvaja-Zannier is to relate the derivative of the Betti map to the Kodaira-Spencer map. After some careful computation, they got a \textbf{sufficient} condition for \eqref{EqQuesACZ} in terms of the derivations on the base \cite[Corollay~2.2.2]{ACZBetti} (which we call \textit{Condition ACZ}). This condition does not depend on $\xi$, so itself is of independent interest. Then they proved Condition ACZ, thus gave an affirmative answer to Question~\ref{ConjBettiRank}, in \textit{loc.cit.} when $\iota_S$ is quasi-finite and $g \le 3$. The real hyperelliptic case, which goes beyond Conjecture~\ref{ConjBettiRank}, is also discussed in Appendix~I of \textit{loc.cit.}

Inspired by several discussions with Andr\'{e}-Corvaja-Zannier, the author applied the pure Ax-Schanuel theorem \cite{MokAx-Schanuel-for} to prove in Appendix of \textit{loc.cit.} that Condition ACZ holds if $\iota_S$ is quasi-finite, $\dim \iota_S(S) \ge g$, and $\iota_S(S)$ is Hodge generic in $\A_g$. In particular, a result of Andr\'{e}-Corvaja-Zannier \cite[Theorem~8.1.1]{ACZBetti} then gives an affirmative answer to Question~\ref{ConjBettiRank} when $\iota_S$ is quasi-finite, $\dim \iota_S(S) \ge g$, and $\mathrm{End}(\cA/S) = \Z$\footnote{By writing $\mathrm{End}(\cA/S)$, we allow finite coverings of $S$.}; see \cite[Theorem~2.3.2]{ACZBetti}. However there are examples with $\iota_S$ quasi-finite, $\dim \iota_S(S) \ge g$ but Condition ACZ violated (most of them arise from Shimura varieties of PEL type), and hence it is hardly possible to  prove the fully answer Question~\ref{ConjBettiRank} using only the pure Ax-Schanuel theorem.

In this paper, we give a different approach to study Question~\ref{ConjBettiRank}, by presenting a simple application of Theorem~\ref{ThmASUnivAbVarIntro} to Conjecture~\ref{ConjBettiRank}. We restrict ourselves to an easy case and prove the following result (Theorem~\ref{ThmAppToBettiRank}).
\begin{thm}
The equality \eqref{EqQuesACZ} holds if the geometric generic fiber of $\cA/S$ is a simple abelian variety, $\iota_S$ is quasi-finite, and $\dim \iota_S(S) \ge g$.
\end{thm}
Note that this theorem and the result of the ACZ paper do not imply each other. 
We also point out that no new contribution to Condition ACZ is made by this method.

\subsection*{A finiteness result \`{a} la Bogomolov}
Bogomolov \cite[Theorem~1]{BogomolovPoints-of-finit} proved the following finiteness result. Let $A$ be an abelian variety over $\C$ and let $Y \subseteq A$ be an irreducible subvariety, then there are finitely many abelian subvarieties $B$ of $A$ with $\dim B > 0$ satisfying: $x + B \subseteq Y$ for some $x \in A(\C)$, maximal for this property.

This finiteness property was extended by Ullmo to pure Shimura varieties \cite[Th\'{e}or\`{e}me~4.1]{UllmoQuelques-applic}, using o-minimal geometry, as an application of the pure Ax-Lindemann theorem. Later on it was extended by the author to mixed Shimura varieties \cite[Theorem~12.2]{GaoTowards-the-And} with a similar proof. The corresponding objects of $x+B$ in the Shimura case are the so-called \textit{weakly special subvarieties} defined by Pink \cite[Definition~4.1.(b)]{PinkA-Combination-o}, which are precisely the bi-algebraic subvarieties (see \cite{UllmoA-characterisat}, \cite[Corollary~8.3]{GaoTowards-the-And}). This finiteness result is useful for the proof of the Andr\'{e}-Oort conjecture.

On the other hand, in order to study the Zilber-Pink conjecture, which is a generalization of the Andr\'{e}-Oort conjecture, Habegger-Pila \cite{HabeggerO-minimality-an} introduced the notion of \textit{weakly optimal} subvarieties of a given subvariety of a mixed Shimura variety; see Definition~\ref{DefnWeaklyOptimalSubvarieties}. They also proved in \textit{loc.cit.} the natural generalization of the finiteness result for weakly optimal subvarieties in the cases of complex abelian varieties and $Y(1)^N$ (product of modular curves). This result is later on generalized by Daw-Ren to any pure Shimura variety \cite[Proposition~3.3]{DawRenAppOfAS}. A key point to pass from Ullmo's result to Daw-Ren's result is to apply the Ax-Schanuel theorem in lieu of the Ax-Lindemann theorem.

In this paper, we prove the corresponding finiteness result for weakly optimal subvarieties for $\mathfrak{A}_g$ (Theorem~\ref{ThmFinitenessAlaBogomolov}). The proof follows the guideline of Daw-Ren by plugging in the author's previous work on extending Ullmo's finiteness result to the mixed case. Denote by $(P_{2g,\mathrm{a}},\cX_{2g,\mathrm{a}}^+)$ the connected mixed Shimura datum of Kuga type associated with $\mathfrak{A}_g$; see $\mathsection$\ref{SubsectionUnivAbVar} for the notation.
\begin{thm}\label{ThmFinitenessAlaBogomolovIntro}
There exists a finite set $\Sigma$ consisting of elements of the form $\left( (Q,\cY^+),N \right)$, where $(Q,\cY^+)$ is a connected mixed Shimura subdatum of $(P_{2g,\mathrm{a}},\cX_{2g,\mathrm{a}}^+)$ and $N$ is a normal subgroup of $Q$ whose reductive part is semi-simple, such that the following property holds. If a closed irreducible subvariety $Z$ of $Y$ is weakly optimal, then there exists $\left( (Q,\cY^+),N \right) \in \Sigma$ such that $Z^{\biZar} = \bu(N(\R)^+\tilde{y})$ for some $\tilde{y} \in \cY^+$.
\end{thm}

\subsection*{Outline of the paper}
In $\mathsection$\ref{SectionUnivAbVar} we introduce, in an example based way, the basic knowledge of $\mathfrak{A}_g$ as a connected mixed Shimura variety. In the end of the section we give the definition and some basic properties (including the realization of the uniformizing space) of connected mixed Shimura varieties of Kuga type.

In $\mathsection$\ref{SectionStatementOfAS} we set up the framework for proving the mixed Ax-Schanuel theorem. We consider not only $\mathfrak{A}_g$, but also all connected mixed Shimura varieties of Kuga type.\footnote{In fact this makes the last step of the proof ($\mathsection$\ref{SectionEndOfProof}) easier.} In particular we review the bi-algebraic system associated with $\mathfrak{A}_g$ and the geometric/group-theoretic interpretation of bi-algebraic subvarieties of $\mathfrak{A}_g$.

The proof of the mixed Ax-Schanuel theorem is done in $\mathsection$\ref{SectionSetup}-$\mathsection$\ref{SectionEndOfProof}. In $\mathsection$\ref{SectionSetup}, 
we fix the basic setup and summarize some  results for the pure part. In particular we cite the volume bounds for pure Shimura varieties. In $\mathsection$\ref{SectionQStabBig} we do the necessary d\'{e}vissage and reduce the mixed Ax-Schanuel theorem to the case where we have a big $\Q$-stabilizer. Apart from the proof of Theorem~\ref{ThmBignessOfQStab}, this section is a standard argument. Then in $\mathsection$\ref{SectionQStabNormal} we prove that the $\Q$-stabilizer can be assumed to be normal in $P$. Here $\mathsection$\ref{SubsectionAlgFamilyAssoWithB}-\ref{SubsectionFirstStepTowardsNormality} are simply the argument of \cite{MokAx-Schanuel-for} adapted to the mixed case, although unlike the pure case we need to be careful with the monodromy group and the Mumford-Tate group. The argument of $\mathsection$\ref{SubsectionLastStepTowardsNormality} is new. 
Then in $\mathsection$\ref{SectionEndOfProof} we finish the proof.

The finiteness result \textit{\`{a} la Bogomolov} is proven in $\mathsection$\ref{SectionAppToAFinitenessResult}. The application to the rank of the Betti map is presented in $\mathsection$\ref{SectionAppToBettiRank}.

At this stage it is worth making some extra comments. The extension of the pure Ax-Schanuel theorem to the mixed one in this paper does \textbf{not} follow the same guideline as the author's previous work on the extension of Ax-Lindemann \cite{GaoTowards-the-And}. In both cases one studies some complex analytic irreducible subset $\tilde{Z}$ in the uniformizing space. For the Ax-Lindemann theorem $\tilde{Z}$ is assumed to be algebraic. So its pure part, being also algebraic, must hit the boundary of the bounded symmetric domain. Thus its pure part is open in a compact set. As we have the freedom to choose $\tilde{Z}$ to have relative dimension $0$ over its pure part in the proof of the Ax-Lindemann theorem, the ``vertical'' direction is uniformly bounded. Thus in \cite[$\mathsection$9-10]{GaoTowards-the-And}, no extra estimate beyond the pure part is needed, and it can be shown that the reductive part of the $\Q$-stabilizer is \textit{a priori} as big as possible. However for the Ax-Schanuel theorem, the pure part of $\tilde{Z}$ does not necessarily hit the boundary, so in order to get any meaningful estimate we need to compare the growth of $\tilde{Z}$ in the \textit{vertical} direction with its growth in the \textit{horizontal} direction. It is impossible to prove the bigness of the reductive part of the $\Q$-stabilizer directly. To solve this problem, we argue as in $\mathsection$\ref{SubsectionLastStepTowardsNormality}. This subsection is not needed to prove the pure Ax-Schanuel theorem because for a pure Shimura datum $(G,\cX^+_G)$, any normal subgroup $N$ of $G$ induces a decomposition of $(G,\cX^+_G)$.

As for the comparison of the growth of $\tilde{Z}$ in the two directions, we have to consider two cases:
the \textit{vertical} direction of $\tilde{Z}$ grows at most polynomially in terms of its growth in the \textit{horizontal} direction, and vice-versa. The former case can be settled by estimates on the pure part without much effort. To solve the latter case, we use a variant of Tsimerman's idea \textit{polynomial for free in unipotent groups} in \cite[pp.~3]{TsimermanAx-Schanuel-and}. See the proof of Theorem~\ref{ThmBignessOfQStab} for more details.

\subsection*{Acknowledgements}
The author would like to thank Ngaiming Mok, Jonathan Pila and Jacob Tsimerman for numerous valuable discussions on the subject and for sending their preprint on the pure Ax-Schanuel theorem, and especially Jacob Tsimerman for pointing out the variant of the  polynomial-for-free argument in Theorem~\ref{ThmBignessOfQStab}. Part of the work was done when the author was a long-term visitor at the Fields Institute under the \textit{Thematic Program on Unlikely Intersections, Heights, and Efficient Congruencing}. The author would like to thank the organizing committee for the invitation and the program for the financial support. The author would also like to particularly thank Ngaiming Mok for his great help on complex geometry and his invitation to the University of Hong Kong. The author benefited a lot from the stimulating atmosphere. The author would like to thank Daniel Bertrand, Jacob Tsimerman and Umberto Zannier for their comments, and especially Ngaiming Mok for pointing out a serious mistake in a previous version of the paper. The author would like to thank the anonymous referees for their careful reading and suggestions to improve the paper.

\section{Universal abelian variety}\label{SectionUnivAbVar}
We recall in this section some basic knowledge of mixed Shimura varieties of Kuga type. In particular we explain how the universal abelian variety fits in this language. In the end we fix some notation for the paper.

\subsection{Moduli space of abelian varieties}
Let $g \ge 1$ be an integer. Let $D = \mathrm{diag}(d_1,\ldots,d_g)$ with $d_1|\cdots|d_g$ be positive integers. Let $N \ge 1$ be an integer. 
Let $\A_{g,D}(N)$ be the moduli space of abelian varieties of dimension $g$ which are polarized of type $D$ equipped with level-$N$-structures.

It is well-known that $\A_{g,D}$ is a connected pure Shimura variety, associated with the connected pure Shimura datum $(\GSp_{2g,D},\mathfrak{H}_g^+)$; see \cite[$\mathsection$1.2]{GenestierNgo}. We hereby give a quick summary of this fact.

Let $V_{2g}$ be a $\Q$-vector space of dimension $2g$, and let 
\begin{equation}\label{EqAltForm}
\Psi \colon V_{2g} \times V_{2g} \rightarrow \G_{a,\Q}, \quad (v_1,v_2) \mapsto v_1^{\!^{\intercal}} \begin{pmatrix} 0 & D \\ -D & 0 \end{pmatrix} v_2
\end{equation}
be a non-degenerate alternating form. Then the $\Q$-group $\GSp_{2g,D}$ is defined by
\begin{align*}
\GSp_{2g,D} & =\{h \in \GL(V_{2g}) : \Psi(hv,hv') = \nu(h)\Psi(v,v')\text{ for some }\nu(h) \in \G_m\} \\
& =\left\{ h \in \GL_{2g} : h \begin{pmatrix} 0 & D \\ -D & 0 \end{pmatrix} h^{\!^{\intercal}} = \nu(h) \begin{pmatrix} 0 & D \\ -D & 0 \end{pmatrix},~ \nu(h) \in \G_m \right\}.
\end{align*}
Let $\Sp_{2g,D} = \GSp_{2g,D}^{\mathrm{der}}$.

Let $\mathfrak{H}_g^+$ be the Siegel upper half-space
\[
\{ Z = X + \sqrt{-1}Y \in M_{g \times g}(\C) : Z = Z^{\!^{\intercal}},~ Y > 0\}.
\]
Then $\GSp_{2g,D}(\R)^+$ acts on $\mathfrak{H}_g^+$ by the formula
\[
\begin{pmatrix} A' & B' \\ C' & D' \end{pmatrix} Z = (A'Z+B')(C'Z+D')^{-1}, \quad \forall \begin{pmatrix} A' & B' \\ C' & D' \end{pmatrix} \in \GSp_{2g,D}(\R)^+\text{ and }Z \in \mathfrak{H}_g^+.
\]
It is known that the action of $\GSp_{2g,D}(\R)^+$ on $\mathfrak{H}_g^+$ thus defined is transitive.

The natural inclusion $\mathfrak{H}_g^+ \subseteq \{ Z = X + \sqrt{-1}Y \in M_{g \times g}(\C) : Z = Z^{\!^{\intercal}}\} \cong \C^{g(g+1)/2}$ realizes $\mathfrak{H}_g^+$ as an open (in the usual topology) semi-algebraic subset of $\C^{g(g+1)/2}$. Hence this inclusion endows $\mathfrak{H}_g^+$ with a complex structure. 

Let $\Gamma_{\Sp_{2g,D}}(N)= \{ h \in \GSp_{2g,D}(\Z): h \equiv I_{2g} \pmod{N} \}$. Then $\A_{g,D}(N) \cong \Gamma_{\Sp_{2g,D}}(N) \backslash \mathfrak{H}_g^+$ as complex varieties. Thus we obtain a uniformization in the category of complex varieties
\begin{equation}\label{EqUniformizationModuliSpace}
\mathfrak{H}_g^+ \rightarrow \A_{g,D}(N).
\end{equation}

\subsection{Universal abelian variety}\label{SubsectionUnivAbVar}
Use the notation of the previous subsection. If furthermore $N\ge 3$, then $\A_{g,D}(N)$ is a fine moduli space and hence admits a universal family, which we call $\mathfrak{A}_{g,D}(N)$. We use $\pi \colon \mathfrak{A}_{g,D}(N) \rightarrow \A_{g,D}(N)$ to denote the natural projection.

The variety $\mathfrak{A}_{g,D}(N)$ is an example of a connected mixed Shimura variety (of Kuga type). We hereby give a construction of the connected mixed Shimura datum of Kuga type $(P_{2g,D,\mathrm{a}},\cX_{2g,\mathrm{a}}^+)$ associated with $\mathfrak{A}_{g,D}(N)$.

Recall that $V_{2g}$ is a $\Q$-vector space of dimension $2g$. By abuse of notation we also use it to denote the $\Q$-vector group of dimension $2g$. Then the natural action of $\GSp_{2g,D}$ on $V_{2g}$ defines a $\Q$-group
\[
P_{2g,D,\mathrm{a}} = V_{2g} \rtimes \GSp_{2g,D}.
\]
It is not hard to see that the unipotent radical $\cR_u(P_{2g,D,\mathrm{a}})$ of $P_{2g,D,\mathrm{a}}$ is $V_{2g}$, and the reductive part $P_{2g,D,\mathrm{a}}/\cR_u(P_{2g,D,\mathrm{a}})$ of $P_{2g,D,\mathrm{a}}$ is $\GSp_{2g,D}$.

The space $\cX_{2g,\mathrm{a}}^+$ is constructed as follows.
\begin{enumerate}
\item[(i)] As a set, $\cX_{2g,\mathrm{a}}^+ = V_{2g}(\R) \times \mathfrak{H}_g^+$.
\item[(ii)] The action of $P_{2g,D,\mathrm{a}}(\R)^+$ on $\cX_{2g,\mathrm{a}}^+$ is defined as follows: for any $(v,h) \in P_{2g,D,\mathrm{a}}(\R)^+ = V_{2g}(\R) \rtimes \GSp_{2g,D}(\R)^+$ and any $(v',x) \in \cX_{2g,\mathrm{a}}^+$, we have
\begin{equation}\label{EqActionOfP2gOnX2g}
(v,h) \cdot (v',x) = (v+hv',hx).
\end{equation}
This action is transitive.
\item[(iii)] Fix a Lagrangian decomposition of $V_{2g}(\R) = \R^{2g} \cong \R^g \times \R^g$. The complex structure of $\cX_{2g,\mathrm{a}}^+$ is the one given by the pullback of the following map
\begin{equation}\label{EqComplexStrOfX2g}
\begin{array}{cccc}
\cX_{2g,\mathrm{a}}^+ = & \R^g \times \R^g \times \mathfrak{H}_g^+ & \xrightarrow{\sim} & \C^g \times \mathfrak{H}_g^+, \\
& (a,b,Z) & \mapsto & (Da+Zb, Z)
\end{array}.
\end{equation}
\end{enumerate}

The pair $(P_{2g,D,\mathrm{a}},\cX_{2g,\mathrm{a}}^+)$ is a connected mixed Shimura datum as defined by Pink; see \cite[Construction~2.9, Example~2.12]{PinkA-Combination-o} and \cite[2.25]{PinkThesis}. There is a natural morphism
\[
\tilde{\pi} \colon (P_{2g,D,\mathrm{a}},\cX_{2g,\mathrm{a}}^+) \rightarrow (\GSp_{2g,D},\mathfrak{H}_g^+)
\]
induced by the projection $P_{2g,D,\mathrm{a}} = V_{2g} \rtimes \GSp_{2g,D} \rightarrow \GSp_{2g,D}$. This is a Shimura morphism as defined by Pink \cite[Definition~2.5]{PinkA-Combination-o} or \cite[2.3]{PinkThesis}.

Let $\Gamma_{P_{2g,D,\mathrm{a}}}(N) = N V(\Z) \rtimes \Gamma_{\Sp_{2g,D}}(N)$. Then as complex varieties we have $\mathfrak{A}_g(N) \cong \Gamma_{P_{2g,D,\mathrm{a}}}(N) \backslash \cX_{2g,\mathrm{a}}^+$. This makes $\mathfrak{A}_{g,D}(N)$ a connected mixed Shimura variety. See \cite[Example~2.12]{PinkA-Combination-o} or \cite[10.5, 10.9, 10.10]{PinkThesis}. We thus obtain a uniformization in the category of complex varieties
\begin{equation}\label{EqUnifUniversalAbVar}
\cX_{2g,\mathrm{a}}^+ \rightarrow \mathfrak{A}_{g,D}(N).
\end{equation}
Now we have the following commutative diagram
\begin{equation}\label{DiagramUnivAbVarAndModuliSpace}
\xymatrix{
(P_{2g,D,\mathrm{a}},\cX_{2g,\mathrm{a}}^+) \ar[r]^-{\tilde{\pi}} \ar[d]_{\eqref{EqUnifUniversalAbVar}} & (\GSp_{2g,D},\mathfrak{H}_g^+) \ar[d]^{\eqref{EqUniformizationModuliSpace}} \\
\mathfrak{A}_{g,D}(N) \ar[r]^{\pi} & \A_{g,D}(N)
}
\end{equation}
where for simplicity of notation the vertical maps are the uniformizations we discussed above.

\subsection{A Hodge-theoretic point of view on $\cX_{2g,\mathrm{a}}^+$}\label{SubsectionHodgeTheoreticX2ga}
Let us take a closer look at $\cX_{2g,\mathrm{a}}^+$. We work with the category of analytic objects in this section. For future purpose, we make the situation slightly more general. Let $\Delta$ be a simply-connected complex space, and let $\pi_{\Delta} \colon \cA_{\Delta} \rightarrow \Delta$ be a polarized family of abelian varieties of dimension $g \ge 1$. Use $\psi$ to denote this polarization.


Let $\underline{\mathbb{Z}}_{\Delta}$ (resp. $\underline{\mathbb{C}}_{\Delta}$) be the constant local system of $\Z$-rank $1$ (resp. of $\C$-dimension $1$)
 on $\Delta$. Let $\underline{\mathbb{Z}}_{\cA_\Delta}$ (resp. $\underline{\mathbb{C}}_{\cA_\Delta}$) be the constant local system of $\Z$-rank $1$ (resp. of $\C$-dimension $1$) on $\cA_\Delta$.

Then $R_1(\pi_{\Delta})_*\underline{\mathbb{Z}}_{\cA_\Delta}$, which we define as the dual of $R^1(\pi_{\Delta})_*\underline{\mathbb{Z}}_{\cA_\Delta}$, is a local system on $\Delta$ such that $(R_1(\pi_{\Delta})_*\underline{\mathbb{Z}}_{\cA_\Delta})_s = H_1(\cA_s,\mathbb{Z})$ for any $s \in \Delta$. It defines a variation of $\mathbb{Z}$-Hodge structures over $\Delta$ of type $\{(0,-1),(-1,0)\}$ in the following way: the $\mathcal{O}_{\Delta}$-sheaf $\mathcal{V}_{\Delta}=R_1(\pi_{\Delta})_*\underline{\mathbb{Z}}_{\cA_\Delta} \otimes_{\underline{\mathbb{Z}}_{\Delta}} \mathcal{O}_{\Delta} \cong \mathcal{H}^1_{\mathrm{dR}}(\cA_{\Delta}/\Delta)^\vee$ is locally free, and has a locally free $\mathcal{O}_{\Delta}$-subsheaf $\mathcal{F}^0\mathcal{V}_{\Delta}$ defined by $(\mathcal{F}^0\mathcal{V}_{\Delta})_s = H^{0,1}(\cA_s)^\vee$ for any $s\in \Delta$. From this we see that $\mathcal{V}_{\Delta}/\mathcal{F}^0\mathcal{V}_{\Delta} \cong (\Omega^1_{\cA_{\Delta}/\Delta})^\vee$. Note that $R_1(\pi_{\Delta})_*\underline{\mathbb{Z}}_{\cA_\Delta}$ and $\mathcal{F}^0\mathcal{V}_{\Delta}$ are subsheaves of $\mathcal{V}_{\Delta}$, and
\begin{equation}\label{EquationTrivialIntersectionSheaves}
R_1(\pi_{\Delta})_*\underline{\mathbb{Z}}_{\cA_\Delta} \cap \mathcal{F}^0\mathcal{V}_{\Delta} \text{ is trivial, namely is the constant sheaf }\underline{0}_{\Delta}.
\end{equation}

Let $\mathbf{V}_{\Delta}$ and $\mathbf{F}^0\mathbf{V}_{\Delta}$ be the vector bundles associated with the locally free sheaves $\mathcal{V}_{\Delta}$ and $\mathcal{F}^0\mathcal{V}_{\Delta}$. Then the discussion in the previous paragraph yields an exact sequence of holomorphic vector bundles on $\Delta$
\begin{equation}\label{EquationRelativeLieAlgebra}
0 \rightarrow \mathbf{F}^0\mathbf{V}_{\Delta} \rightarrow \mathbf{V}_{\Delta} \rightarrow \mathrm{Lie}(\cA_{\Delta}/\Delta) \rightarrow 0.
\end{equation}

Next we explain that there exists a real analytic diffeomorphism $\mathrm{Lie}(\cA_{\Delta}/\Delta) \cong \R^{2g} \times \Delta$.



Note that $R_1(\pi_{\Delta})_*\underline{\mathbb{C}}_{\cA_\Delta} = R_1(\pi_{\Delta})_*\underline{\mathbb{Z}}_{\cA_\Delta} \otimes_{\underline{\mathbb{Z}}_{\Delta}} \underline{\mathbb{C}}_{\Delta}$ is the local system on $\Delta$ of the $\mathbb{C}$-dimension $2g$. There is a unique vector bundle with flat connection associated with it. Since $\Delta$ is simply connected, this vector bundle with flat connection is trivial. In other words we have $\mathbf{V}_{\Delta} \cong \mathbb{C}^{2g} \times \Delta$ and $R_1(\pi_{\Delta})_*\underline{\mathbb{C}}_{\cA_\Delta}$ is the subsheaf of $\mathcal{V}_{\Delta}$ consisting of constant sections (recall that $\mathcal{V}_{\Delta}$ is the sheaf of sections of the vector bundle $\mathbf{V}_{\Delta}$ over $\Delta$). So the constant sheaf $R_1(\pi_{\Delta})_*\underline{\mathbb{Z}}_{\cA_\Delta}$ is the subsheaf of $\mathcal{V}_{\Delta}$ consisting of the sections of $\mathbf{V}_{\Delta}$ over $\Delta$ with image in $\mathbb{Z}^{2g} \times \mathcal{V}_{\Delta}$ (these sections are constant since $\mathbb{Z}^{2g}$ is discrete). So the inclusion of sheaves $R_1(\pi_{\Delta})_*\underline{\mathbb{Z}}_{\cA_\Delta} \subseteq \mathcal{V}_{\Delta}$ corresponds to the natural inclusion $\mathbb{Z}^{2g} \times \Delta \subseteq \mathbb{C}^{2g} \times \Delta \cong \mathbf{V}_{\Delta}$. The intersection of $\mathbb{Z}^{2g} \times \Delta$ and $\mathbf{F}^0\mathbf{V}_{\Delta}$ in $\mathbf{V}_{\Delta}$ is the zero section of $\mathbf{V}_{\Delta} \rightarrow \Delta$ by \eqref{EquationTrivialIntersectionSheaves}.

Thus we obtain an injective map $\mathbb{Z}^{2g} \times \Delta \rightarrow \mathrm{Lie}(\cA_{\Delta}/\Delta)$ as the composite
\[
\mathbb{Z}^{2g} \times \Delta \subseteq \mathbb{C}^{2g} \times \Delta \cong \mathbf{V}_{\Delta} \rightarrow \mathbf{V}_{\Delta}/\mathbf{F}^0\mathbf{V}_{\Delta} \cong \mathrm{Lie}(\cA_{\Delta}/\Delta).
\]
See \eqref{EquationRelativeLieAlgebra} for the last equality. This injective map extends to
\begin{equation}\label{EquationBettiFirstStep}
\mathbb{R}^{2g} \times \Delta \subseteq \mathbb{C}^{2g} \times \Delta \cong \mathbf{V}_{\Delta}  \rightarrow  \mathbf{V}_{\Delta} / \mathbf{F}^0\mathbf{V}_{\Delta}  \cong \mathrm{Lie}(\cA_{\Delta}/\Delta),
\end{equation}
where the first inclusion is obtained by identifying $\mathbb{R}^{2g}$ as the real part of $\mathbb{C}^{2g}$. Note that apart from the first inclusion, every morphism in this composition is holomorphic.

Denote by
\begin{equation}\label{EqiDelta}
i_{\Delta} \colon \R^{2g} \times \Delta \rightarrow \mathrm{Lie}(\cA_{\Delta}/\Delta)
\end{equation}
the composite of the maps in \eqref{EquationBettiFirstStep}. Then the following properties  clearly hold true.
\begin{enumerate}
\item[(i)] The map $i_{\Delta}$ is a real analytic diffeomorphism. It is semi-algebraic if $\Delta$ has a semi-algebraic structure;
\item[(ii)] For any $a \in \mathbb{R}^{2g}$, the image $i_{\Delta}(\{a\} \times \Delta)$ is complex analytic;
\item[(iii)] Over each $s \in \Delta$, the induced map $(i_{\Delta})_s \colon \mathbb{R}^{2g} \rightarrow \mathrm{Lie}(\cA_{\Delta}/\Delta)_s$ is a group homomorphism.
\end{enumerate}

We end this subsection by relating it to the last subsection. Take $\Delta = \mathfrak{H}_g^+$ and $\cA_{\Delta}$ to be the pullback of $\mathfrak{A}_{g,D}(N) \rightarrow \A_{g,D}(N)$ under the uniformization $\mathfrak{H}_g^+ \rightarrow \A_{g,D}(N)$ in \eqref{EqUniformizationModuliSpace}. Denote by $\cX_{2g,\mathrm{a}}^+ := \mathrm{Lie}(\cA_{\Delta}/\Delta)$ for this choice. The complex space $\cX_{2g,\mathrm{a}}^+$ thus obtained coincides with the one defined in $\mathsection$\ref{SubsectionUnivAbVar}. Indeed, fix a Lagrangian decomposition $V_{2g}(\R) = \R^{2g} \cong \R^g \times \R^g$ and write
\begin{equation}\label{EqTwoComplexStrOfX2g}
\begin{array}{cccc}
\cX_{2g,\mathrm{a}}^+ \xleftarrow[i_{\mathfrak{H}_g^+}]{\sim} & V_{2g}(\R) \times \mathfrak{H}_g^+ = \R^g \times \R^g \times \mathfrak{H}_g^+ & \xrightarrow{\sim} & \C^g \times \mathfrak{H}_g^+ \\
& (a,b,Z) & \mapsto & (Da+Zb, Z)
\end{array}.
\end{equation}
From this we obtain $\rho_g \colon \cX_{2g,\mathrm{a}}^+ \rightarrow \C^g \times \mathfrak{H}_g^+$. Then it is not hard to see that $\rho_g$ is an isomorphism of complex spaces. Since $i_{\mathfrak{H}_g^+}$ is semi-algebraic, we can conclude that the $\cX_{2g,\mathrm{a}}^+$ defined above is the same complex space as the one defined in $\mathsection$\ref{SubsectionUnivAbVar} in view of \eqref{EqComplexStrOfX2g}.

\subsection{Mixed Shimura variety of Kuga type}
In this subsection we recall the definition and some basic properties of mixed Shimura varieties of Kuga type. Let $\S = \mathrm{Res}_{\C/\R}\G_{m,\C}$ be the Deligne torus.
\begin{defn}[{\cite[Definition~2.1]{PinkA-Combination-o}}]\label{DefnMixedShimuraDatum}
A \textbf{connected mixed Shimura datum of Kuga type} is a pair $(P,\cX^+)$ where
\begin{itemize}
\item $P$ is a connected linear algebraic group over $\Q$ whose unipotent radical $V$ is a vector group,
\item $\cX^+ \subseteq \hom(\S,P_{\R})$ is a left homogeneous space under $P(\R)^+$,
\end{itemize}
such that for one (and hence for all) $x \in \cX^+$, we have
\begin{enumerate}
\item[(i)] the adjoint representation induces on $\lie P$ a rational mixed Hodge structure of type
\[
\{(-1,1),(0,0),(1,-1)\} \cup \{(-1,0),(0,-1)\},
\]
\item[(ii)] the weight filtration on $\lie P$ is given by
\[
W_n(\lie P) = \begin{cases}  0 & \text{if }n < -1 \\ \lie V & \text{if } n = -1 \\ \lie P & \text{if } n \ge 0 \end{cases},
\]
\item[(iii)] the conjugation by $x(\sqrt{-1})$ induces a Cartan involution on $G^{\mathrm{ad}}_{\R}$ where $G = P/V$, and $G^{\mathrm{ad}}$ posses no $\Q$-factor $H$ such that $H(\R)$ is compact,
\item[(iv)] the group $P/P^{\mathrm{der}} = Z(G)$ is an almost direct product of a $\Q$-split torus with a torus of compact type defined over $\Q$,
\item[(v)] $P$ possesses no proper normal subgroup $P'$ such that $x$ factors through $P'_{\R} \subseteq P_{\R}$.
\end{enumerate}
If in addition $P$ is reductive, then $(P,\cX^+)$ is called a \textbf{connected pure Shimura datum}.
\end{defn}

\begin{rmk}
\begin{enumerate}
\item By our convention, every connected pure Shimura datum is a connected mixed Shimura datum of Kuga type. The pair $(P_{2g,D,\mathrm{a}},\cX_{2g,\mathrm{a}}^+)$ in $\mathsection$\ref{SubsectionUnivAbVar} is an example of a connected mixed Shimura datum of Kuga type, which is not pure.
\item Conditions (i) and (ii) imply that $\dim V$ is even.
\item Definition~\ref{DefnMixedShimuraDatum} is a particular case of \cite[Definition~2.1]{PinkA-Combination-o}. The difference is that the subgroup $U$ in $\cR_u(P)$ in \textit{loc.cit.} is trivial in our case (connected mixed Shimura datum of Kuga type), hence property (i) of Definition~2.1 of \textit{loc.cit.} is not needed here.
\item Condition (iv) of Definition~\ref{DefnMixedShimuraDatum} implies that every sufficiently small congruence subgroup $\Gamma$ of $P(\Q)$ is contained in $P^{\mathrm{der}}(\Q)$; see \cite[proof of 3.3(a)]{PinkThesis}. Fix a Levi decomposition $P = V\rtimes G$, then $P^{\mathrm{der}} = V \rtimes G^{\mathrm{der}}$, and hence any congruence subgroup $\Gamma < P^{\mathrm{der}}(\Q)$ is Zariski dense in $P^{\mathrm{der}}$ by condition (iii); see \cite[Theorem~4.10]{AlgebraicGroupBible}.
\item Condition (v) is what Pink calls ``irreducible'' in \cite[2.13]{PinkThesis}. It means precisely that $P$ is the Mumford-Tate group of a generic choice of $x \in \cX^+$. For the purpose of studying transcendence results or unlikely intersections, it is harmless to consider only this kind of connected mixed Shimura data. Hence we put this condition in the definition as Pink does in \cite{PinkA-Combination-o}.
\item The space $\cX^+$ has a unique structure of complex manifold such that for every representation $\rho \colon P \rightarrow \GL(W)$, the Hodge filtration determined by $\rho \circ x$ varies holomophically with $x \in \cX^+$. In particular this complex structure is invariant under $P(\R)^+$. See \cite[Fact~2.3(b)]{PinkA-Combination-o} or \cite[1.18]{PinkThesis}. We will also give more details on this in $\mathsection$\ref{SubsectionRealizationOfX}.
\end{enumerate}
\end{rmk}

Let $(P,\cX^+)$ be a connected mixed Shimura datum of Kuga type. 
The following results are proven by Pink; see \cite[Fact~2.3(c)-(e)]{PinkA-Combination-o} or \cite[3.3, 9.24]{PinkThesis}. For any congruence subgroup $\Gamma \subseteq P(\Q) \cap P(\R)_+$, where $P(\R)_+$ is the stabilizer of $\cX^+ \subseteq \hom(\S,P_{\R})$, the quotient $\Gamma\backslash\cX^+$ is a complex analytic space with at most finite quotient singularities, and has a natural structure of a quasi-projective algebraic variety over $\C$. Moreover it is smooth if $\Gamma$ is sufficiently small.

\begin{defn}
A \textbf{connected mixed Shimura variety of Kuga type} $M$ associated with $(P,\cX^+)$ is the quotient $\Gamma \backslash \cX^+$ from above. Then we have a uniformization in the category of complex analytic spaces $\bu \colon \cX^+ \rightarrow M$.
\end{defn}

\begin{defn}\label{DefnShimuraMorphism}
\begin{enumerate}
\item[(i)] A \textbf{(Shimura) morphism} of connected mixed Shimura data of Kuga type $\tilde{\psi} \colon (Q,\cY^+) \rightarrow (P,\cX^+)$ is a homomorphism $\tilde{\psi} \colon Q \rightarrow P$ of algebraic groups over $\Q$ which induces a map $\cY^+ \rightarrow \cX^+$, $y \mapsto \tilde{\psi} \circ y$. 
\item[(ii)] In particular, if $\tilde{\psi}$ is an inclusion (on the group and the underlying space), then we say that $(Q,\cY^+)$ is a \textbf{connected mixed Shimura subdatum} of $(P,\cX^+)$.
\item[(iii)] A \textbf{Shimura morphism} of connected mixed Shimura varieties of Kuga type  is a morphism of algebraic varieties induced by a Shimura morphism of connected mixed Shimura data of Kuga type.
\item[(iv)] In particular given $M$ a connected mixed Shimura variety of Kuga type associated with $(P,\cX^+)$, the subvarieties of $M$ coming from connected mixed Shimura subdata of $(P,\cX^+)$ are called \textbf{connected mixed Shimura subvarieties}.
\end{enumerate}
\end{defn}
Pink proved that every Shimura morphism of connected mixed Shimura varieties of Kuga type is algebraic. See \cite[Facts~2.6]{PinkA-Combination-o} or \cite[3.4, 9.24]{PinkThesis}.

The morphism $\tilde{\pi}$ in \eqref{DiagramUnivAbVarAndModuliSpace} is an example of Shimura morphism of connected mixed Shimura data of Kuga type. It induces a Shimura morphism $\pi$ of connected mixed Shimura varieties of Kuga type.

Before going on, let us make the following remark. The notion ``connected mixed Shimura subdatum'' exists \textit{a priori} for any connected mixed Shimura datum $(P,\cX^+)$, not necessarily of Kuga type. However if $(P,\cX^+)$ is of Kuga type, then all its connected mixed Shimura subdata \textit{a priori} defined are again of Kuga type by reason of weight; see \cite[Proposition~2.9]{GaoTowards-the-And}. Hence our Definition~\ref{DefnShimuraMorphism}(ii) is compatible with the usual convention.

The following fact is proven by Pink \cite[2.9]{PinkThesis}.
\begin{fact}
Let $(P,\cX^+)$ be a connected mixed Shimura datum of Kuga type. Let $P_0$ be a normal subgroup of $P$. Then there exists a quotient connected mixed Shimura datum of Kuga type $(P,\cX^+)/P_0$ and a Shimura morphism $(P,\cX^+) \rightarrow (P,\cX^+)/P_0$, unique up to isomorphism, such that every Shimura morphism $(P,\cX^+) \rightarrow (P',\cX^{\prime+})$, where the homomorphism $P \rightarrow P'$ factors through $P/P_0$, factors in a unique way through $(P,\cX^+)/P_0$. 
\end{fact}
In fact Pink proved that such a quotient exists for an arbitrary connected mixed Shimura datum (not necessarily of Kuga type), and then it is not hard to see that the resulting $(P,\cX^+)/P_0$ is of Kuga type if $(P,\cX^+)$ is.

\begin{notation}\label{NotationMSV}
We use the following notation. Let $M = \Gamma \backslash \cX^+$ be a connected mixed Shimura variety of Kuga type associated with $(P,\cX^+)$, and let $\bu \colon \cX^+ \rightarrow M$ be the uniformization.

We use $V$ to denote $\cR_u(P)$, and $(G,\cX_G^+)$ to denote $(P,\cX^+)/V$. The quotient Shimura morphism is denoted by $\tilde{\pi} \colon (P,\cX^+) \rightarrow (G,\cX_G^+)$.

We use $M_G$ to denote the connected pure Shimura variety $\tilde{\pi}(\Gamma) \backslash \cX_G^+$, and use $\bu_G \colon \cX_G^+ \rightarrow M_G$ to denote the uniformization. Then $\tilde{\pi}$ induces a Shimura morphism $\pi \colon M \rightarrow M_G$. We thus have the following commutative diagram:
\begin{equation}\label{DiagramKugaTypeToPure}
\xymatrix{
(P,\cX^+) \ar[r]^-{\tilde{\pi}} \ar[d]_{\bu} & (G,\cX_G^+) \ar[d]^{\bu_G} \\
M \ar[r]^-{\pi} & M_G.
}
\end{equation}
\end{notation}

We gave an example of a connected mixed Shimura varieties of Kuga type, namely the universal abelian variety; see $\mathsection$\ref{SubsectionUnivAbVar}. In particular \eqref{DiagramUnivAbVarAndModuliSpace} is a particular case of \eqref{DiagramKugaTypeToPure}. But in fact, all connected mixed Shimura varieties of Kuga type arise from this case. More precisely, Pink proved the following result, which he called \textit{reduction lemma}.\footnote{We only need the result for Kuga type, so the statement is simpler.}
\begin{thm}[{\cite[2.26]{PinkThesis}}]\label{ThmReductionLemma}
Let $(P,\cX^+)$ be a connected mixed Shimura datum of Kuga type such that $\dim V = 2g > 0$, where $V = \cR_u(P)$. Then there exist a connected pure Shimura datum $(G_0,\cD^+)$, a diagonal matrix $D = \mathrm{diag}(d_1,\ldots,d_g)$ with $d_1|\cdots |d_g$ positive integers, and a Shimura morphism
\[
\tilde{\lambda} \colon (P,\cX^+) \hookrightarrow (G_0,\cD^+) \times (P_{2g,D,\mathrm{a}}, \cX_{2g,\mathrm{a}}^+).
\]
\end{thm}

Taking the quotient of both sides by their unipotent radicals, we obtain a Shimura morphism $\tilde{\lambda}_G \colon (G,\cX^+_G) \rightarrow (G_0,\cD^+) \times (\GSp_{2g,D},\mathfrak{H}_g^+)$; see \cite[Proposition~2.9]{GaoTowards-the-And} for a proof.

Theorem~\ref{ThmReductionLemma} has the following immediate corollary.
\begin{cor}
Let $M = \Gamma \backslash \cX^+$ be a connected mixed Shimura variety of Kuga type associated with $(P,\cX^+)$. Assume $\dim V = 2g$. Then up to replacing $\Gamma$ by a subgroup of finite index, we have that $\pi \colon M \rightarrow M_G$ is an abelian scheme of relative dimension $g$.
\end{cor}
\begin{proof}
If $V = 0$, then the conclusion certainly holds. Assume $\dim V = 2g > 0$. Then apply Theorem~\ref{ThmReductionLemma} to $(P,\cX^+)$. We have the following cartesian diagram
\[
\xymatrix{
(P,\cX^+) \ar[r]^-{\tilde{\lambda}} \ar[d]_{\tilde{\pi}} \pullbackcorner & (G_0,\cD^+) \times (P_{2g,D,\mathrm{a}}, \cX_{2g,\mathrm{a}}^+) \ar[d] \\
(G,\cX_G^+) \ar[r]^-{\tilde{\lambda}_G} & (G_0,\cD^+) \times (\GSp_{2g,D},\mathfrak{H}_g^+).
}
\]
It is cartesian by comparing the dimension of the vertical fibers.

Now up to replacing $\Gamma$ by a subgroup of finite index, we may assume that $\tilde{\lambda}$ induces a closed immersion $M \rightarrow M_{G_0} \times \mathfrak{A}_g(N)$ by \cite[3.8(b)]{PinkThesis}. Thus we obtain a cartesian diagram
\[
\xymatrix{
M \ar[r] \ar[d]_{\pi} \pullbackcorner & M_{G_0} \times \mathfrak{A}_{g,D}(N) \ar[d] \\
M_G \ar[r] & M_{G_0} \times \A_{g,D}(N)
}
\]
where the horizontal morphisms are closed immersions. Hence $\pi \colon M \rightarrow M_G$ is an abelian scheme of relative dimension $g$.
\end{proof}

\subsection{Realization of $\cX^+$}\label{SubsectionRealizationOfX}
Let $(P,\cX^+)$ be a connected mixed Shimura datum of Kuga type. We recall the realization of $\cX^+$ as in \cite[$\mathsection$4]{GaoTowards-the-And}.

We start with the dual $\cX^\vee$ of $\cX^+$; see \cite[1.7(a)]{PinkThesis} or \cite[Chapter~VI, Proposition~1.3]{MilneCanonical-model}.

Let $W$ be a faithful representation of $P$ and take any point $x_0 \in \cX^+$. The weight filtration on $W$ is constant, so the Hodge filtration $x \mapsto \mathrm{Fil}_x^{\cdot}(W_{\C})$ gives an injective map $\cX^+ \hookrightarrow \mathrm{Grass}(W)(\C)$ to a certain flag variety. In fact, the injective map factors through
\[
\cX^+ = P(\R)^+/\mathrm{Stab}_{P(\R)^+}(x_0) \hookrightarrow P(\C)/\exp(\mathrm{Fil}_{x_0}^0 \mathrm{Lie}P_{\C}) \hookrightarrow \mathrm{Grass}(W)(\C),
\]
where the first injection is an open immersion; see \cite[1.7(a)]{PinkThesis} or \cite[Chapter~VI, (1.2.1)]{MilneCanonical-model}. We define the dual $\cX^\vee$ of $\cX^+$ to be
\[
\cX^\vee = P(\C)/\exp(\mathrm{Fil}_{x_0}^0 \mathrm{Lie}P_{\C}).
\]
Then $\cX^\vee$ is clearly a connected smooth complex algebraic variety.

\begin{prop}[{\cite[Proposition~4.1 and Remark~4.4]{GaoTowards-the-And}}]\label{PropositionRealizationOfX}
Under the open (in the usual topology) immersion $\cX^+ \hookrightarrow \cX^\vee$, the space $\cX^+$ is realized as a semi-algebraic subset which is also a complex manifold. In particular, the complex structure of any $\cX^+_{x_G}$ ($x_G \in \cX_G^+$) is the same as the one obtained from $\cX^+_{x_G} \cong V(\C) / \mathrm{Fil}_{x_G}^0 V(\C)$.
\end{prop}

\begin{eg} Let us look at the example $(P_{2g,D,\mathrm{a}},\cX_{2g,\mathrm{a}}^+)$. In this case, we can take $W$ be to a $\Q$-vector space of dimension $2g+1$ and identify $P_{2g,D,\mathrm{a}}$ with the following subgroup of $\GL_{2g+1}$:
\[
\begin{pmatrix}
\GSp_{2g,D} & V_{2g} \\
0 & 1
\end{pmatrix}.
\]
It is then not hard to prove that the complex structure of $\cX_{2g,\mathrm{a}}^+$ given by Proposition~\ref{PropositionRealizationOfX} coincides with the one given by \eqref{EqComplexStrOfX2g}.\footnote{The realization of $\mathfrak{H}_g^+$ by Proposition~\ref{PropositionRealizationOfX} is the Harish-Chandra realization, whereas the one given by \eqref{EqComplexStrOfX2g} is the Siegel upper half-space realization. It is known that for $\mathfrak{H}_g^+$, the complex structures given by these two realizations coincide. Moreover, the various realizations of $\cX_G^+$ give the same semi-algebraic structure by \cite[Lemma~2.1]{UllmoQuelques-applic}.}
\end{eg}

The identification \eqref{EqComplexStrOfX2g} can be generalized to arbitrary $(P,\cX^+)$ of Kuga type in the following way. Use Notation~\ref{NotationMSV}. By \cite[pp. 6]{WildeshausThe-canonical-c}, there exists a Shimura morphism $\tilde{i} \colon (G,\cX_G^+) \rightarrow (P,\cX^+)$ such that $\tilde{\pi} \circ \tilde{i} = \mathrm{id}$. Then $\tilde{i}$ defines a Levi decomposition of $P = V\rtimes G$. Recall $\cX^+ \subseteq \hom(\S,P_{\R})$. Define the bijective map
\[
\rho \colon V(\R) \times \cX_G^+ \xrightarrow{\sim} \cX^+, \quad (v,x) \mapsto \mathrm{int}(v) \circ \tilde{i}(x).
\]
Under this identification, the action of $P(\R)^+ = V(\R) \rtimes G(\R)^+$ on $\cX^+$ is given by the formula
\[
(v',h)\cdot (v,x) = (v' + hv , hx), \quad \forall (v',h) \in P(\R)^+\text{ and }(v,x) \in \cX^+.
\]
It is proven that for the realization of $\cX^+$ as an open (in the usual topology) semi-algebraic subset of $\cX^\vee$, the identification $\rho$ above is semi-algebraic; see \cite[Proposition~4.3]{GaoTowards-the-And}.

\section{Statement of the Ax-Schanuel theorem}\label{SectionStatementOfAS}
\subsection{Review on the bi-algebraic system}\label{SubsectionReviewOnBiAlgSystem}
Let $M = \Gamma \backslash \cX^+$ be a connected mixed Shimura variety of Kuga type associated with $(P,\cX^+)$. Let $\bu \colon \cX^+ \rightarrow M$ be the uniformization.

Recall the realization of $\cX^+$ in $\mathsection$\ref{SubsectionRealizationOfX}. By Proposition~\ref{PropositionRealizationOfX} $\cX^+$ can be realized as an open (in the usual topology) semi-algebraic subset of a complex algebraic variety $\cX^\vee$.

\begin{defn}
\begin{enumerate}
\item[(i)] A subset $\tilde{Y}$ of $\cX^+$ is said to be \textbf{irreducible algebraic} if it is a complex analytic irreducible component of $\cX^+ \cap W$, where $W$ is an algebraic subvariety of $\cX^\vee$.
\item[(ii)] A subset $\tilde{Y}$ of $\cX^+$ is said to be \textbf{irreducible bi-algebraic} if it is an irreducible algebraic subset of $\cX^+$ and $\bu(\tilde{Y})$ is an algebraic subvariety of $M$.
\item[(iii)] A closed irreducible subvariety $Y$ of $M$ is said to be \textbf{bi-algebraic} if one (and hence any) complex analytic irreducible subvariety of $\bu^{-1}(Y)$ is irreducible algebraic in $\cX^+$. A closed subvariety $Y$ of $M$ is said to be bi-algebraic if all of its irreducible components are bi-algebraic.
\end{enumerate}
\end{defn}
The following result is not hard to prove.
\begin{lemma}
If $F_1$ and $F_2$ are bi-algebraic subvarieties of $M$, then every irreducible component of $F_1 \cap F_2$ is also bi-algebraic.
\end{lemma}
In view of this, we can introduce the following notation.
\begin{notation}\label{NotationBiAlg}
\begin{enumerate}
\item[(i)] Let $\tilde{Z}$ be any complex analytic irreducible subset of $\cX^+$. Use $\tilde{Z}^{\Zar}$ to denote the smallest irreducible algebraic subset of $\cX^+$ which contains $\tilde{Z}$, and use $\tilde{Z}^{\biZar}$ to denote the smallest irreducible bi-algebraic subset of $\cX^+$ which contains $\tilde{Z}$.
\item[(ii)] Let $Z$ be any subset of $M$ (not necessarily a subvariety). Use $Z^{\biZar}$ to denote the smallest bi-algebraic subvariety of $M$ which contains $Z$ (hence $Z^{\biZar}$ contains $Z^{\Zar}$).
\end{enumerate}
\end{notation}

\subsection{Statement of the Ax-Schanuel theorem}
Let $M = \Gamma \backslash \cX^+$ be a connected mixed Shimura variety of Kuga type associated with $(P,\cX^+)$. Let $\bu \colon \cX^+ \rightarrow M$ be the uniformization. The typical case is when $M= \mathfrak{A}_{g,D}$ is the universal abelian variety.

\begin{thm}[Ax-Schanuel for mixed Shimura varieties of Kuga type]\label{ThmASUnivAbVar}
Let $\mathscr{Z}$ be a complex analytic irreducible subvariety of $\mathrm{graph}(\bu) \subseteq \cX^+ \times M$, and denote by $\tilde{Z}$ the image of $\mathscr{Z}$ under the natural projection $\cX^+ \times M \rightarrow \cX^+$. Then
\[
\dim \mathscr{Z}^{\Zar} - \dim \tilde{Z} \ge \dim \tilde{Z}^{\biZar},
\]
where $\mathscr{Z}^{\Zar}$ means the Zariski closure of $\mathscr{Z}$ in $M \times \cX^+$. Moreover the equality holds if and only if $\tilde{Z}$ is a complex analytic irreducible component of $\tilde{Z}^{\Zar} \cap \bu^{-1}(\bu(\tilde{Z})^{\Zar})$. 
\end{thm}

We also state the weak Ax-Schanuel theorem for mixed Shimura varieties of Kuga type to make the statement more clear. It follows directly from Theorem~\ref{ThmASUnivAbVar}.
\begin{thm}[weak Ax-Schanuel for mixed Shimura varieties of Kuga type]\label{ThmWASUnivAbVar}
Let $\tilde{Z}$ be a complex analytic irreducible subset of $\cX^+$. Then
\[
\dim (\bu(\tilde{Z}))^{\Zar} + \dim \tilde{Z}^{\Zar} \ge \dim \tilde{Z} + \dim \tilde{Z}^{\biZar}.
\]
Moreover the equality holds if and only if $\tilde{Z}$ is a complex analytic irreducible component of $\tilde{Z}^{\Zar} \cap \bu^{-1}(\bu(\tilde{Z})^{\Zar})$.
\end{thm}

Before moving on, we point out that the ``Moreover'' part of Theorem~\ref{ThmWASUnivAbVar} immediately follows from the main part and the Intersection Dimension Inequality.

\begin{rmk} 
Our proof of Theorem~\ref{ThmASUnivAbVar} uese the work of Mok-Pila-Tsimerman \cite{MokAx-Schanuel-for} on the Ax-Schanuel theorem for pure Shimura varieties and extends their proof. As a statement itself, Theorem~\ref{ThmASUnivAbVar} implies the pure Ax-Schanuel theorem.
\end{rmk}

\subsection{Geometric description of bi-algebraic subvarieties}\label{SubsectionGeomDescriptionBiAlg}
By \cite[Theorem~8.1]{GaoTowards-the-And}, bi-algebraic subvarieties of $\mathfrak{A}_{g,D}$ are precisely the weakly special subvarieties defined by Pink \cite[Definition~4.1.(b)]{PinkA-Combination-o}. The definition of weakly special subvarieties will be given in $\mathsection$\ref{SubsectionWeaklySpecial}. For the moment we present the geometric description of these weakly special subvarieties. Recall the projection $\pi \colon \mathfrak{A}_{g,D} \rightarrow \A_{g,D}$, which is proper. We have the following result.
\begin{prop}[{\cite[Proposition~1.1]{GaoA-special-point}}]\label{PropGeomDescriptionOfBiAlgebraic}
A closed irreducible subvariety $Y$ of $\mathfrak{A}_{g,D}$ is bi-algebraic (or equivalently weakly special) if and only if the following conditions hold:
\begin{enumerate}
\item[(i)] Its projection $\pi(Y)$ is a weakly special subvariety of $\A_{g,D}$;
\item[(ii)] Up to taking a finite covering of $\pi(Y)$, we have that $Y$ is the translate of an abelian subscheme of $\pi^{-1}(\pi(Y)) \rightarrow \pi(Y)$ by a torsion section and then by a constant section.
\end{enumerate}
\end{prop}
Let us explain condition~(ii) in more details. We have that $\pi(Y)$ is a closed irreducible subvariety of $\A_{g,D}$, hence $\pi^{-1}(\pi(Y)) = \mathfrak{A}_{g,D}|_{\pi(Y)}$ is an abelian scheme over $\pi(Y)$. Condition~(ii) means: there exists a finite covering $B' \rightarrow \pi(Y)$ such that under the base change $\mathfrak{A}' := B' \times_{\pi(Y)} \pi^{-1}(\pi(Y))$ and the natural projection $p' \colon \mathfrak{A}' \rightarrow \pi^{-1}(\pi(Y))$, we have $Y = p'(\mathfrak{B}+\sigma + \sigma_0)$, where $\mathfrak{B}$ is an abelian subscheme of $\mathfrak{A}' \rightarrow B'$, $\sigma$ is a torsion section of $\mathfrak{A}'\rightarrow B'$, and $\sigma_0$ is a constant section of (the largest constant abelian subscheme of) $\mathfrak{A}'\rightarrow B'$.

Proposition~\ref{PropGeomDescriptionOfBiAlgebraic} is proven as a consequence of the following result.
\begin{prop}[{\cite[Proposition~3.3]{GaoA-special-point}}]\label{PropGeomOfsgInUnivAbVar}
Let $B$ be an irreducible subvariety of $\A_{g,D}$. Then
\[
\begin{array}{c}
\{\text{Up to taking a finite covering of }B,\text{ the translates of an abelian subscheme} \\ 
\text{of }\pi^{-1}(B) \rightarrow B\text{ by a torsion section and then by a constant section}\} \\
= \{\text{irreducible components of }\pi^{-1}(B) \cap F : F\text{ weakly special in }  \mathfrak{A}_{g,D} \text{ with } B \subseteq \pi(F)\}.
\end{array}
\]
\end{prop}


\subsection{Review on weakly special subvarieties}\label{SubsectionWeaklySpecial}
Now let us give the definition of weakly special subvarieties following Pink. Let $(P,\cX^+)$ and $M$ be as above Theorem~\ref{ThmASUnivAbVar}.
\begin{defn}
\begin{enumerate}
\item[(i)] A subset $\tilde{Y}$ of $\cX^+$ is said to be \textbf{weakly special} if there exist a connected mixed Shimura subdatum of Kuga type $(Q,\cY^+)$ of $(P,\cX^+)$, a normal subgroup $N$ of $Q^{\der}$, and a point $\tilde{y} \in \cY^+$ such that $\tilde{Y} = N(\R)^+\tilde{y}$.
\item[(ii)] A subvariety $Y$ of $M$ is said to be \textbf{weakly special} if $Y = \bu(\tilde{Y})$ for some weakly special subset $\tilde{Y}$ of $\cX^+$.
\end{enumerate}
\end{defn}
Our formulation is slightly different from \cite[Definition~4.1.(b)]{PinkA-Combination-o}, but it is not hard to show that they are equivalent; see \cite[$\mathsection$5.1]{GaoTowards-the-And}.

\section{Basic Setting-up}\label{SectionSetup}
In this section, we fix some basic setting-up to prove Theorem~\ref{ThmASUnivAbVar}. Let $M$ be a connected mixed Shimura variety of Kuga type associated with $(P,\cX^+)$. Let $\bu \colon \cX^+ \rightarrow M$ be the uniformization. Use $\Delta \subseteq \cX^+ \times M$ to denote $\mathrm{graph}(\bu)$. Let $\tilde{\pi} \colon (P,\cX^+) \rightarrow (G,\cX^+_G)$ and $\pi \colon M \rightarrow M_G$ be as in Notation~\ref{NotationMSV}.

Let $\mathscr{Z} = \mathrm{graph}(\tilde{Z} \rightarrow \bu(\tilde{Z}))$ be a complex analytic irreducible subset of $\mathrm{graph}(\bu)$. Let $\mathscr{Z}^{\Zar}$ be the smallest algebraic subvariety of $\cX^+ \times M$ containing $\mathscr{Z}$. We wish to prove
\[
\dim \mathscr{Z}^{\Zar} - \dim \mathscr{Z} \ge \dim \tilde{Z}^{\biZar}.
\]
It is clear that we may replace $\mathscr{Z}$ by a complex analytic irreducible component of $\mathscr{Z}^{\Zar} \cap \Delta$. Hence Theorem~\ref{ThmASUnivAbVar} is equivalent to the following statement.
\begin{thm}\label{ThmEquivalentFormOfAS}
Let $\mathscr{B}$ be an irreducible algebraic subvariety of $\cX^+ \times M$, and let $\mathscr{Z}$ be a complex analytic irreducible component of $\mathscr{B} \cap \Delta$. Assume $\mathscr{B} = \mathscr{Z}^{\Zar}$. Then
\[
\dim \mathscr{B} - \dim \mathscr{Z} \ge \dim \tilde{Z}^{\biZar},
\]
where $\tilde{Z}$ is the image of $\mathscr{Z}$ under the natural projection $\cX^+ \times M \rightarrow \cX^+$.
\end{thm}
We shall prove the Ax-Schanuel theorem in the form of Theorem~\ref{ThmEquivalentFormOfAS}. Since any connected mixed Shimura subvariety of $M$ is bi-algebraic, we may replace $(P,\cX^+)$ by the smallest connected mixed Shimura datum such that $\tilde{Z} \subseteq \cX^+$ and replace $M$ accordingly. Then we still have $\mathscr{B} \subseteq \cX^+ \times M$. Use
\[
\mathrm{pr}_{\cX^+} \colon \cX^+ \times M \rightarrow \cX^+, \quad \mathrm{pr}_M \colon \cX^+ \times M \rightarrow M
\]
to denote the natural projections. They are clearly algebraic. We still have $\tilde{Z} = \mathrm{pr}_{\cX^+}(\mathscr{Z})$.

Consider the action of $\Gamma$ on $\cX^+ \times M$ via its action on the first factor. Then $\Delta$ is $\Gamma$-invariant.

\subsection{Fundamental Set}\label{SubsectionFundamentalSet}

Fix a Shimura embedding $(P,\cX^+) \hookrightarrow (G_0,\cD^+) \times (P_{2g,D,\mathrm{a}}, \cX_{2g,\mathrm{a}}^+)$. Such an embedding exists by Theorem~\ref{ThmReductionLemma}.

Let $\mathfrak{F}_{\Sp_{2g}}$ be a Siegel fundamental set for the action of $\Sp_{2g,D}(\Z)$ on $\mathfrak{H}_g^+$. Let $\mathfrak{F}_{P_{2g,\mathrm{a}}} \subseteq \cX_{2g,\mathrm{a}}^+$ be defined as $i_{\mathfrak{H}_g^+}\left( (-k,k)^{2g} \times \mathfrak{F}_{\Sp_{2g}} \right)$ for some $k \ge 1$, where
\[
i_{\mathfrak{H}_g^+} \colon V_{2g}(\R) \times \mathfrak{H}_g^+ \cong \cX_{2g,\mathrm{a}}^+
\]
is the real-algebraic map defined in \eqref{EqTwoComplexStrOfX2g}. Then $\mathfrak{F}_{P_{2g,\mathrm{a}}}$ is a fundamental set for the action of $\Z^{2g} \rtimes \Sp_{2g,D}(\Z)$ on $\cX_{2g,\mathrm{a}}^+$.

Let $\mathfrak{F}_{\cD^+}$ be a fundamental set for the action of $\Gamma_0$ on $\cD^+$ as in \cite[Theorem~1.9]{KlinglerThe-Hyperbolic-}. Then 
$\mathfrak{F} = \mathfrak{F}_{\cD^+} \times \mathfrak{F}_{P_{2g,\mathrm{a}}}$ 
is a fundamental set for $\bu \colon \cD^+ \times \cX_{2g,\mathrm{a}}^+ \rightarrow S_0 \times \mathfrak{A}_g$. Moreover $\bu|_{\mathfrak{F}}$ is definable in $\R_{\mathrm{an},\exp}$ by \cite{PeterzilDefinability-of}.

It is possible to choose $\mathfrak{F}_{\Sp_{2g}}$, $k$ and $\mathfrak{F}_{\cD^+}$ such that $\mathfrak{F}\cap \cX^+$ is a fundamental set for $\bu \colon \cX^+ \rightarrow M$. Replace $\mathfrak{F}$ by $\mathfrak{F} \cap \cX^+$, then $\bu|_{\mathfrak{F}}$ is definable in $\R_{\mathrm{an},\exp}$. We may furthermore enlarge $\mathfrak{F}$ such that $\mathfrak{F}$, resp. $\mathfrak{F}_G:= \tilde{\pi}(\mathfrak{F})$, is open in $\cX^+$, resp. in $\cX_G^+$, in the usual topology.


\subsection{Some results on the pure part}
Let $\tilde{Z}$ be as in Theorem~\ref{ThmEquivalentFormOfAS}. In this subsection we summarize some known results on the pure part, which will be used in the proof of Theorem~\ref{ThmBignessOfQStab}. Assume $\dim \tilde{\pi}(\tilde{Z}) > 0$. All distances and norms $|| \cdot ||_{\infty}$ below are as in \cite[$\mathsection$5.1]{KlinglerThe-Hyperbolic-}.

Fix $\tilde{z}_{0,G} \in \tilde{\pi}(\tilde{Z})$. All constants below depend only on $\cX^+_G$, $\mathfrak{F}_G$, $\tilde{\pi}(\tilde{Z})$, and $\tilde{z}_{0,G}$.

For any $T > 0$, define
\begin{equation}\label{EqGeodesicBallPurePart}
B^{\mathrm{horz}}(\tilde{z}_{0,G}, T) = \{ \tilde{z}_G \in \cX^+_G : d^{\mathrm{horz}}(\tilde{z}_G, \tilde{z}_{0,G}) < \log T \},
\end{equation}
where $d^{\mathrm{horz}}(\cdot, \cdot)$ is the $G^{\mathrm{der}}(\R)^+$-invariant hyperbolic metric on $\cX^+_G$.

\begin{lemma}\label{LemmaInftyNormPointInTheBase}
For each $\tilde{z}_G \in B^{\mathrm{horz}}(\tilde{z}_{0,G}, T) \subseteq \cX^+_G$, we have
\begin{equation*}\label{EqInftyNormPointInTheBase}
|| \tilde{z}_G ||_{\infty} \le c_0 T.
\end{equation*}
\end{lemma}
\begin{proof}
For $\tilde{z}_G \in B^{\mathrm{horz}}(\tilde{z}_{0,G}, T) \subseteq \cX^+_G$, there exists $g \in G(\R)^+$ such that $g\cdot \tilde{z}_{0,G} = \tilde{z}_G$. Then by \cite[Lemma~5.4]{KlinglerThe-Hyperbolic-}, we have
\[
\log || g ||_{\infty} \le d^{\mathrm{horz}}(\tilde{z}_G, \tilde{z}_{0,G}) < \log T.
\]
Hence we are done.
\end{proof}

Now let us consider
\begin{equation}\label{EqGammaGIntersectionNonEmpty}
L_G(T):=\{ \gamma_G \in \Gamma_G : \gamma_G \mathfrak{F}_G \cap B^{\mathrm{horz}}(\tilde{z}_{0,G}, T) \not= \emptyset \}.
\end{equation}

\begin{lemma}\label{LemmaGammaGUpperBound}
There exist two constants $c_1, c_2 >0$ such that
\begin{equation*}\label{EqGammaGUpperBound}
H(\gamma_G) \le c_1 T^{c_2}\quad \text{ for any }\quad \gamma_G \in L_G(T).
\end{equation*}
\end{lemma}
\begin{proof}
For any $\gamma_G \in L_G(T)$, we have $\gamma_G^{-1}\tilde{z}_G \in \mathfrak{F}_G$ for some $\tilde{z}_G \in B^{\mathrm{horz}}(\tilde{z}_{0,G},T)$. 
Thus the result follows from Lemma~\ref{LemmaInftyNormPointInTheBase} and \cite[Lemma~5.5]{KlinglerThe-Hyperbolic-}.
\end{proof}

All volumes below are hyperbolic volumes.

\begin{lemma}\label{LemmaVolumeLowerBoundPurePart}
There exist two constants $c_3 >0$ and $c_4 > 0$ such that
\begin{equation*}\label{EqVolumeLowerBoundPurePart}
\mathrm{vol}(\tilde{\pi}(\tilde{Z}) \cap B^{\mathrm{horz}}(\tilde{z}_{0,G},T)) \ge c_3 T^{c_4 \dim \tilde{\pi}(\tilde{Z})}, \quad \forall T \gg 0.
\end{equation*}
\end{lemma}
\begin{proof}
This is \cite[Corollary~3]{HwangVolumes-of-comp}.
\end{proof}

\begin{lemma}\label{LemmaVolumeUpperBoundPurePart}
There exists a constant $c_5 > 0$ such that 
\begin{equation*}\label{EqVolumeUpperBoundPurePart}
\mathrm{vol}(\tilde{\pi}(\tilde{Z}) \cap \gamma_G\mathfrak{F}_G) \le c_5
\end{equation*}
for all $\gamma_G \in \Gamma_G$. 
\end{lemma}
\begin{proof}
This follows from \cite[Proposition~3.2]{BakkerThe-Ax-Schanuel}; our $\tilde{\pi}(\tilde{Z}) \cap \gamma_G\mathfrak{F}_G$ is the image of the $Z \cap \gamma \Phi$ in \textit{loc.cit.} under the projection $\cX^+_G \times M_G \rightarrow \cX^+_G$. As our notation is somewhat different from \textit{loc.cit.}, we briefly recall the proof. Denote by $\tilde{X} = \tilde{Z}^{\Zar}$ and $Y = \bu(\tilde{Z})^{\Zar}$. The assumption of Theorem~\ref{ThmEquivalentFormOfAS} ($\mathscr{Z}$ is a complex analytic irreducible component of $\mathscr{B} \cap \Delta$) implies that $\tilde{Z}$ is a complex analytic irreducible component of $\tilde{X} \cap \bu^{-1}(Y)$. Hence $\tilde{\pi}(\tilde{Z})$ is a complex analytic irreducible component of $\tilde{\pi}(\tilde{X}) \cap \bu_G^{-1}(\pi(Y))$, where $\bu_G \colon \cX^+_G \rightarrow M_G$ is the uniformizing map. It is possible to cover $\mathfrak{F}_G$ with finitely many semi-algebraic subsets $\{\Sigma_i\}$ such that each $\Sigma_i$ can be written in terms of (Siegel) coordinates; see \cite[Lemma~5.8]{KlinglerThe-Hyperbolic-} or \cite[Proposition~3.2]{BakkerThe-Ax-Schanuel}. For each dominant projection $p$ from $\gamma_G^{-1}\tilde{\pi}(\tilde{Z}) \cap \Sigma_i$ to $\dim \tilde{\pi}(\tilde{Z})$ coordinates, it can be computed that $p(\gamma_G^{-1}\tilde{\pi}(\tilde{Z}) \cap \mathfrak{F}_G \cap \Sigma_i)$ has finite volume. As the K\"{a}hler form with respect to which we compute the volume is $G^{\der}(\R)^+$-invariant, it suffices to bound the degree of the projections uniformly for $\gamma_G \in \Gamma_G$. But the function $G(\R)^+ \rightarrow \R$, $g^{-1}\tilde{\pi}(\tilde{X}) \cap (\bu_G^{-1}(\pi(Y)) \cap \mathfrak{F}_G) \cap \Sigma_i \mapsto \deg(p)$ is a definable function with value in $\Z$. Hence the image must be a finite set, meaning that the degree is uniformly bounded for $g \in G(\R)^+$, and in particular for $\gamma_G \in \Gamma_G$. Hence we are done.
\end{proof}

\section{Bigness of the $\Q$-stabilizer}\label{SectionQStabBig}

Define $H$ to be the $\Q$-stabilizer of $\mathscr{B}$, namely
\begin{equation}\label{EqDefnQStabilizerOfB}
H = \left( \bar{\Gamma \cap \mathrm{Stab}_{P(\R)^+}(\mathscr{B})}^{\Zar} \right)^{\circ}.
\end{equation}
The goal of this section is to execute the following d\'{e}vissage.
\begin{prop}\label{PropBignessOfQStabilizer}
Either Theorem~\ref{ThmEquivalentFormOfAS} is true, or $\dim H > 0$.
\end{prop}

Let us prove this proposition. 
Define
\begin{equation}\label{EquationXiSetsForFiberOfMSV}
\Theta=\{p\in P(\R): \dim(p^{-1}\mathscr{B}\cap  (\mathfrak{F} \times M) \cap \Delta)=\dim \mathscr{Z}\}. 
\end{equation}
Then $\Theta$ is a definable set.

We have
\begin{equation}\label{EquationIntersectionWithLatticeEqualASFiber}
\{\gamma\in\Gamma : \gamma (\mathfrak{F} \times M) \cap \mathscr{Z} \neq\emptyset\} \subseteq \Theta\cap\Gamma
\end{equation}
since $\Gamma \Delta = \Delta$. On the other hand, we have
\begin{align*}
\gamma (\mathfrak{F} \times M) \cap \mathscr{Z} & = (\gamma \cdot \mathrm{pr}_{\cX^+}^{-1}(\mathfrak{F})) \cap \mathscr{Z} \\
& = \mathrm{pr}_{\cX^+}^{-1}(\gamma\mathfrak{F}) \cap \mathscr{Z} \quad \text{ since }\mathrm{pr}_{\cX^+}\text{ is }P(\R)^+\text{-equivariant} \\
& = \mathrm{pr}_{\cX^+}^{-1}(\gamma\mathfrak{F} \cap \mathrm{pr}_{\cX^+}(\mathscr{Z}))  \\
& = \mathrm{pr}_{\cX^+}^{-1}(\gamma\mathfrak{F} \cap \tilde{Z}).
\end{align*}
Hence \eqref{EquationIntersectionWithLatticeEqualASFiber} becomes
\begin{equation}\label{EqIntersectingFundamentalDomainInTheta}
\{ \gamma \in \Gamma : \gamma \mathfrak{F} \cap \tilde{Z} \not= \emptyset \}  \subseteq \Theta \cap \Gamma.
\end{equation}

\begin{thm}\label{ThmBignessOfQStab}
Assume $\dim \tilde{Z} > 0$.
Then there exist a constant $\epsilon > 0$ and a sequence $\{T_i\}$ with $T_i \rightarrow \infty$ such that the following property holds: for each $T_i$ there exists a connected semi-algebraic block $B \subseteq \Theta$ which contains $\ge T_i^{\epsilon}$ points in $\Gamma$ with height at most $T_i$.\footnote{We refer to \cite[Definition~3.4 and the paragraph below]{PilaO-minimality-an} for the definition and basic properties of semi-algebraic blocks.}
\end{thm}
\begin{proof}
Use the notation of \eqref{DiagramKugaTypeToPure}. A typical case to keep in mind is $(P,\cX^+) = (P_{2g,D,\mathrm{a}},\cX^+_{2g,\mathrm{a}})$ and $(G,\cX^+_G) = (\GSp_{2g,D},\mathfrak{H}_g^+)$. The map $\tilde{\pi} \colon (P,\cX^+) \rightarrow (G,\cX^+_G)$ is the natural projection. Denote by $\tilde{X} = \tilde{Z}^{\Zar}$ and $Y= \bu(\tilde{Z})^{\Zar}$. Then the assumption on $\mathscr{Z}$ implies that $\tilde{Z}$ is a complex analytic irreducible component of $\tilde{X} \cap \bu^{-1}(Y)$.

\noindent\boxed{\text{Case }\dim \tilde{\pi}(\tilde{Z}) = 0} In this case $\tilde{Z}$ is contained in a fiber of $\cX^+ \rightarrow \cX^+_G$. Consider $\{ \gamma \in \Gamma : \gamma \mathfrak{F} \cap \tilde{Z} \not= \emptyset \}$. We claim that it is infinite. Assume it is not, then $\bu(\tilde{Z}) = \bigcup_{\gamma \in \Gamma} \bu(\tilde{Z} \cap \gamma\mathfrak{F})$ 
is a finite union, with each member in the union being closed, complex analytic and definable (in $\R_{\mathrm{an},\exp}$) in $M$. Hence $\bu(\tilde{Z})$ is closed complex analytic and definable in $M$. Hence $\bu(\tilde{Z})$ is algebraic by Peterzil-Starchenko's o-minimal Chow \cite[Theorem~4.5]{PeterzilComplex-analyti}; see also \cite[Theorem~2.2]{MokAx-Schanuel-for}, and so $\bu(\tilde{Z}) = Y$. The monodromy group of $Y$, denoted by $\Gamma_Y$, is infinite since $Y$ is a positive dimensional subvariety of an abelian variety. But $\Gamma_Y \subseteq \{ \gamma \in \Gamma : \gamma \mathfrak{F} \cap \tilde{Z} \not= \emptyset \}$. This settles the claim.

Now that $\{ \gamma \in \Gamma : \gamma \mathfrak{F} \cap \tilde{Z} \not= \emptyset \}$ is infinite, we get that it contains $\ge T$ elements of height at most $T$ (for all $T \gg 0$) because each fundamental set in the fiber of $\cX^+ \rightarrow \cX^+_G$ is contained in an Euclidean ball of a fixed radius and that $\tilde{Z}$ is connected.\footnote{The crucial point is that the group $V(\R)$ is a Euclidean space. See \cite[pp.~3]{TsimermanAx-Schanuel-and}.} Hence we can conclude the theorem by Pila-Wilkie \cite[Theorem~3.6]{PilaO-minimality-an}.

\noindent\boxed{\text{Case }\dim \tilde{\pi}(\tilde{Z}) > 0} Fix $\tilde{z}_0 \in \tilde{Z} \cap \mathfrak{F} \subseteq \cX^+$, and denote by $\tilde{z}_{0,G} = \tilde{\pi}(\tilde{z}_0)$. Consider the geodesic balls $B^{\mathrm{horz}}(\tilde{z}_{0,G}, T)$ in $\cX^+_G$ defined by \eqref{EqGeodesicBallPurePart}, which for simplicity we denote by $B^{\mathrm{horz}}(T)$. 

For each $T > 0$, let $\tilde{Z}(T)$ denote the complex analytic irreducible component of $\tilde{Z} \cap \tilde{\pi}^{-1}(B^{\mathrm{horz}}(T)) $ which contains $\tilde{z}_0$. Define the following sets:
\[
\begin{array}{c}
\Xi(T) = \{p \in P(\R)  : p \mathfrak{F} \cap \tilde{Z}(T) \not= \emptyset\}, \\
\Xi_G(T) = \{g \in G(\R): g \mathfrak{F}_G \cap \tilde{\pi}(\tilde{Z}(T))  \not= \emptyset \}.
\end{array}
\]

In the rest of the proof we prove that $\#(\Xi(T) \cap \Gamma)$ grows polynomially in terms of $T$. It is divided into two steps: first we treat the base, and then we insert the information on the vertical direction.

\noindent\boxed{\text{Step~I}} Count on the base.

It is clear that
\[
\tilde{\pi}(\tilde{Z}(T)) = \bigcup_{\gamma_G \in \Xi_G(T) \cap \Gamma_G} (\gamma_G \mathfrak{F}_G \cap \tilde{\pi}(\tilde{Z}(T))).
\]
Thus by taking volumes on both sides, we get by Lemma~\ref{LemmaVolumeLowerBoundPurePart}\footnote{When $T \rightarrow \infty$, we have that $\mathrm{vol}(\tilde{\pi}(\tilde{Z}(T)))$ approximates $\mathrm{vol}(\tilde{\pi}(\tilde{Z}) \cap B^{\mathrm{horz}}(T))$  since $\tilde{Z}$ is irreducible. Thus they have the same volume lower bounds when $T \gg 0$.} and Lemma~\ref{LemmaVolumeUpperBoundPurePart} that
\begin{equation}\label{EqManyPointsInXiG}
\# (\Xi_G(T) \cap \Gamma_G) \ge c T^{\epsilon}, \quad \forall T \gg 0
\end{equation}
for some constants $c > 0$ and $\epsilon > 0$ independent of $T$.

On the other hand for each $\gamma_G \in \Xi_G(T) \cap \Gamma_G$, we have $\gamma_G \in L_G(T)$ for the set $L_G(T)$ defined in \eqref{EqGammaGIntersectionNonEmpty}. Thus Lemma~\ref{LemmaGammaGUpperBound} implies $H(\gamma_G) \le c_1 T^{c_2}$ for some constants $c_1, c_2 > 0$ independent of $T$.

\noindent\boxed{\text{Step~II}} Study the vertical direction.

Consider the projection to the $V$-direction
\[
p_V \colon \cX^+ \cong V(\R) \times \cX^+_G \rightarrow V(\R).
\]
By abuse of notation we shall identify $\cX^+$ and $V(\R) \times \cX^+_G$.

For any set $E \subseteq V(\R)$, denote by $||E||_{\infty} = \max\{||e||_{\infty} : e \in E\}$.


Fix a number $\delta > c_2$. We are in one of the following two cases.
\begin{enumerate}
\item[(i)] For a sequence $\{T_i \in \R \}_{i \in \N}$ such that $T_i \rightarrow \infty$, we have $||p_V(\tilde{Z}(T_i) )||_{\infty} \le T_i^{\delta}$.
\item[(ii)] We have $|| p_V( \tilde{Z}(T) ) ||_{\infty} > T^{\delta}$ for all $T \gg 0$.
\end{enumerate}

Assume we are in case (i). We claim that for each $\gamma_G \in \Xi_G(T_i) \cap \Gamma_G$, there exists $\gamma_V \in \Gamma_V$ with $H(\gamma_V) \le T_i^{\delta}$ such that $(\gamma_V,\gamma_G) \in \Xi(T_i) \cap \Gamma$. Indeed, let $\gamma_G \in \Xi_G(T_i) \cap \Gamma_G$ and choose a point $\tilde{z}_G \in \gamma_G\mathfrak{F}_G \cap \tilde{\pi}(\tilde{Z}(T_i))$. Take $\tilde{z} \in \tilde{Z}(T_i)$ such that $\tilde{\pi}(\tilde{z}) = \tilde{z}_G$. Write $\tilde{z} = (\tilde{z}_V , \tilde{z}_G) \in V(\R) \times \cX^+_G \cong \cX^+$, then by assumption on case (i) we have $||\tilde{z}_V ||_{\infty} \le T_i^{\delta}$. Hence $\tilde{z}_V \in \gamma_V + (-1,1)^{2g}$ with some $\gamma_V \in \Gamma_V$ such that $H(\gamma_V) \le ||\tilde{z}_V ||_{\infty} \le T_i^{\delta}$. Recall that by choice of $\mathfrak{F}$ we have $(-1,1)^{2g} \times \mathfrak{F}_G \subseteq \mathfrak{F}$.  This $\gamma_V$ is what we desire.

Now that $\#(\Xi_G(T_i) \cap \Gamma_G) \ge c T_i^{\epsilon}$ by \eqref{EqManyPointsInXiG} and $H(\gamma_G) \le c_1 T_i^{c_2}$ for each $\gamma_G \in \Xi_G(T_i) \cap \Gamma_G$ (see below \eqref{EqManyPointsInXiG}), the paragraph above yields
\[
\#\{\gamma \in \Xi(T_i) \cap \Gamma : H(\gamma) \le T_i \} \ge c T_i^{\epsilon}
\]
where $c$ and $\epsilon$ are modified appropriately (but still independent of $T_i$). By definition we have $\Xi(T_i) \cap \Gamma \subseteq \{\gamma \in \Gamma : \gamma \mathfrak{F} \cap \tilde{Z} \not= \emptyset\}$, and hence by \eqref{EqIntersectingFundamentalDomainInTheta} we have
\[
\#\{\gamma \in \Theta \cap \Gamma : H(\gamma) \le T_i \} \ge c T_i^{\epsilon}.
\]
Now Pila-Wilkie \cite[Theorem~3.6]{PilaO-minimality-an} yields the conclusion for case (i).

Assume we are in case (ii). Recall that by the choice by $\mathfrak{F}$ we have $\mathfrak{F} \subseteq (-k,k)^{2g} \times \mathfrak{F}_G$ for some fixed integer $k$. For each $\gamma = (\gamma_V, \gamma_G) \in \Gamma$ such that $\gamma_G \mathfrak{F}_G \cap B^{\mathrm{horz}}(T) \not= \emptyset$, we have $H(\gamma_G) \le c_1 T^{c_2}$ by Lemma~\ref{LemmaGammaGUpperBound}. Now for each $\tilde{x} = (\tilde{x}_V, \tilde{x}_G) \in \gamma \mathfrak{F}$, we have
\[
\tilde{x}_V \in \gamma_G \cdot (-k,k)^{2g} + \gamma_V \subseteq (-k c_1 T^{c_2}, k c_1 T^{c_2})^{2g} + \gamma_V.
\]
Thus if $\gamma = (\gamma_V,\gamma_G) \in \Gamma$ satisfies $\gamma\mathfrak{F} \cap \tilde{Z}(T) \not=\emptyset$, then we have
\[
p_V(\gamma\mathfrak{F} \cap \tilde{Z}(T)) \subseteq (-k c_1 T^{c_2}, k c_1 T^{c_2})^{2g} + \gamma_V \subseteq B_{\mathrm{Eucl}}(\sqrt{2g}kc_1T^{c_2}) + \gamma_V,
\]
where $B_{\mathrm{Eucl}}(\sqrt{2g}kc_1T^{c_2})$ is the Euclidean ball in $V(\R)$ centered at $0$ of radius $\sqrt{2g}kc_1T^{c_2}$.



On the other hand consider the Euclidean ball in $V(\R)$ centered at $0$ of radius $T^{\delta}$ for all $T \gg 0$. The assumption of case (ii) says that $p_V(\tilde{Z}(T))$ reaches the boundary of this ball.

Since $\tilde{Z}(T)$ is connected and $V(\R)$ is Euclidean, the two paragraphs above then imply
\begin{equation}\label{EqGammaVPolyGrowth}
\#\{\gamma_V \in \Gamma_V : H(\gamma_V) \le T^{\delta},~(\gamma_V,\gamma_G) \in \Xi(T)\text{ for some }\gamma_G \in \Xi_G(T) \cap \Gamma_G\} \ge \frac{1}{\sqrt{2g}k c_1} T^{\delta - c_2}
\end{equation}
for $T \gg 0$. On the other hand each $\gamma_G \in \Xi_G(T) \cap \Gamma_G$ satisfies $H(\gamma_G) \le c_1 T^{c_2}$. Hence \eqref{EqGammaVPolyGrowth} yields
\[
\#\{(\gamma_V, \gamma_G) \in \Xi(T) \cap \Gamma : H(\gamma_V) \le T^{\delta}, H(\gamma_G) \le c_1 T^{c_2} \} \ge \frac{1}{\sqrt{2g}k c_1} T^{\delta - c_2}.
\]
Recall that the only assumption on $\delta$ is that $\delta > c_2$, so $\delta$ is independent of $T$. Hence we have
\[
\#\{\gamma \in \Xi(T) \cap \Gamma : H(\gamma) \le T \} \ge c' T^{\epsilon}
\]
for some $c' , \epsilon > 0$ independent of $T$. 
Hence by Pila-Wilkie \cite[Theorem~3.6]{PilaO-minimality-an} we are done.
\end{proof}

Now we are ready to prove the bigness of the $\Q$-stabilizer of $\mathscr{B}$.

\begin{proof}[Proof of Proposition~\ref{PropBignessOfQStabilizer}]
Denote by $F = \bu(\tilde{Z})^{\mathrm{biZar}}$ and $\tilde{F} = \tilde{Z}^{\mathrm{biZar}}$. Then $F$ is weakly special by \cite[Theorem~8.1]{GaoTowards-the-And}. As $\tilde{Z}$ is by assumption Hodge generic in $\cX^+$, we have $\tilde{F} = N(\R)^+\tilde{x}$ for some $N \lhd P$ and some $\tilde{x} \in \cX^+$. See $\mathsection$\ref{SubsectionWeaklySpecial}. Denote by $\Gamma_N = N(\Q) \cap \Gamma$. Without loss of generality we may assume $\dim \tilde{Z} > 0$; otherwise Theorem~\ref{ThmEquivalentFormOfAS} is clearly true.

We do the lexicographic induction on $(\dim\mathscr{B}-\dim\mathscr{Z},\dim\mathscr{Z})$, upwards for the first factor and downwards for the second.

The starting point of the induction on the first factor is when $\dim\mathscr{B}-\dim\mathscr{Z}=0$. Then $\mathscr{B}=\mathscr{Z}$ and hence both $\tilde{Z}$ and $\bu(\tilde{Z})$ are algebraic. Therefore $\tilde{Z} = \tilde{F}$ and $\bu(\tilde{Z}) = F$. Thus $\mathscr{Z} = \mathrm{graph}(\tilde{F} \rightarrow F)$. In particular $\mathscr{Z}$ is $\Gamma_N$-invariant, namely $\Gamma_N\cdot \mathscr{Z} =\mathscr{Z}$. But the action of $P(\R)^+$ on $\cX^+ \times M$ is algebraic. So taking Zariski closures on both sides, we get $N(\R)^+\mathscr{B} = \mathscr{B}$. But then $\mathscr{B} = \tilde{F} \times F$, and hence $\dim \mathscr{B} - \dim \mathscr{Z} = \dim F$. So we must have $\dim F = 0$. Hence this case is proven.

Let us prove the starting point of the induction on the second factor. Recall that $\mathscr{Z} \subseteq \mathrm{graph}(\tilde{F} \rightarrow F)$. So the starting point of the induction on the second factor is when $\mathscr{Z} = \mathrm{graph}(\tilde{F} \rightarrow F)$. Hence $\mathscr{Z}$ is $\Gamma_N$-invariant, namely $\Gamma_N\cdot \mathscr{Z} =\mathscr{Z}$. But the action of $P(\R)^+$ on $\cX^+ \times M$ is algebraic. So taking Zariski closures on both sides, we get $N(\R)^+\mathscr{B} = \mathscr{B}$. But then $\mathscr{B} = \tilde{F} \times F$, and hence $\dim \mathscr{B} - \dim \mathscr{Z} = \dim F$. Hence we are done for this case.

Now we do the induction. Let $C$ be a connected semi-algebraic curve in $\Theta$. For each $c$ in an open neighborhood of $C$, let $\mathscr{Z}_c$ be a complex analytic irreducible component of $c^{-1}\mathscr{B}\cap \Delta$ such that $\dim\mathscr{Z}_c=\dim\mathscr{Z}$
. Let $c_0\in C$ be such that $\mathrm{pr}_{\cX^+}(\mathscr{Z}_{c_0})^{\mathrm{biZar}} = \tilde{F}$,\footnote{This holds for all but countably many points in $C$.} then there are 3 possibilities:
\begin{enumerate}
\item[(i)] $c^{-1}\mathscr{B}$ is independent of $c\in C$;
\item[(ii)] $c^{-1}\mathscr{B}$ is not independent of $c\in C$ but $\mathscr{Z}_{c_0}\subseteq c^{-1}\mathscr{B}$ for all $c\in C$;
\item[(iii)] $c^{-1}\mathscr{B}$ is not independent of $c\in C$ and $\mathscr{Z}_{c_0}\not\subseteq (c^\prime)^{-1}\mathscr{B}$ for some $c^\prime\in C$.
\end{enumerate}
If we are in case (ii), then let $\mathscr{B}_1:=\mathscr{B}\cap c^{-1}\mathscr{B}$ for a generic $c\in C$. Then $\dim\mathscr{B}_1<\dim\mathscr{B}$. But $\dim\mathscr{Z}_{c_0}=\dim\mathscr{Z}$. Applying the inductive hypothesis on $\dim\mathscr{B}-\dim\mathscr{Z}$ to $(\mathscr{B}_1,\mathscr{Z}_{c_0})$, we have $\dim\mathscr{B}_1-\dim\mathscr{Z}_{c_0}\ge \dim F$. Hence the Theorem~\ref{ThmEquivalentFormOfAS} is true. If we are in case (iii), then let $\mathscr{B}_2:=(C^{-1}\mathscr{B})^{\Zar}$ and let $\mathscr{Z}_2$ be a complex analytic irreducible component of $\mathscr{B}_2\cap\Delta$. Then $\dim\mathscr{B}_2=\dim\mathscr{B}+1$\footnote{The action of $P(\R)^+$ on $\cX^+ \times M$ extends to an action of $P(\C)$ on $\cX^\vee \times M$. Thus $(C^{-1})^{\Zar}\mathscr{B}$, being the image of $(C^{-1})^{\Zar} \times \mathscr{B}$ under $P(\C) \times (\cX^\vee \times M) \rightarrow \cX^\vee \times M$, has dimension at most $\dim \mathscr{B} + 1$. As $C^{-1}\mathscr{B} \subseteq (C^{-1})^{\Zar}B$ we have $\dim (C^{-1}\mathscr{B})^{\Zar} \le \dim (C^{-1})^{\Zar}B \le \dim \mathscr{B} + 1$. As $C^{-1}\mathscr{B} \subseteq (C^{-1})^{\Zar}B$ we have $\dim (C^{-1}\mathscr{B})^{\Zar} > \dim \mathscr{B}$ by the assumption of (iii).}, and hence $\dim\mathscr{Z}_2=\dim\mathscr{Z}+1$ since $\mathscr{Z}_2$ contains $\cup_{c\in C}\mathscr{Z}_c$. Hence $\dim\mathscr{B}_2-\dim\mathscr{Z}_2=\dim\mathscr{B}-\dim\mathscr{Z}$, and we can apply the inductive hypothesis on $\dim\mathscr{Z}$ to $(\mathscr{B}_2,\mathscr{Z}_2)$ and get $\dim\mathscr{B}_2-\dim\mathscr{Z}_2\ge \dim F$. Hence Theorem~\ref{ThmEquivalentFormOfAS} is true.

It remains to treat case (i). In particular we may take $C$ to be a semi-algebraic curve contained in the semi-algebraic block $B$ as in Theorem~\ref{ThmBignessOfQStab}; note that $B$ depends on the chosen $T_i$ in Theorem~\ref{ThmBignessOfQStab}. It is known that $B$ is the union of all such $C$'s; see \cite[below Definition~3.4]{PilaO-minimality-an}. The assumption on case (i) implies that each $C$ is contained in some left coset of $\stab_{P(\R)}(\mathscr{B})$. Since $B$ is path-connected, it is possible to connect any two points of $B$ by a semi-algebraic curve. So all these $C$'s are contained in the same left coset of $\stab_{P(\R)}(\mathscr{B})$. Thus $B \subseteq p \stab_{P(\R)}(\mathscr{B})$ for some $p \in P(\R)$. In particular $b \in p\stab_{P(\R)}(\mathscr{B})$ for each $b \in B$. Since left cosets are disjoints, we then have $b\stab_{P(\R)}(\mathscr{B}) = p\stab_{P(\R)}(\mathscr{B})$ for each $b \in B$.

So $B \subseteq b \stab_{P(\R)}(\mathscr{B})$ for each $b \in B$. In particular $\gamma^{-1}B \subseteq \stab_{P(\R)}(\mathscr{B})$ for some $\gamma \in B \cap \Gamma$. By letting $T_i \rightarrow \infty$ and varying $B$ accordingly, we get $\dim H > 0$.
\end{proof}

\begin{rmk}\label{RmkAssumptionASEnlarge}
As a byproduct, the proof above yields the following claim: In order to prove the Ax-Schanuel theorem (equivalently Theorem~\ref{ThmEquivalentFormOfAS}), we may assume that every connected semi-algebraic block $B$ of positive dimension in $\Theta$ is contained in a left coset of $\stab_{P(\R)}(\mathscr{B})$.
\end{rmk}

\section{Normality of the $\Q$-stabilizer}\label{SectionQStabNormal}
The goal of this section is to prove the following proposition. Use the notation of $\mathsection$\ref{SectionQStabBig}.
\begin{prop}\label{PropNormOfQStab}
Without loss of generality we may assume $H \lhd P$.
\end{prop}

The proof of Proposition~\ref{PropNormOfQStab} is by upward induction on $\dim \mathscr{B}$. It is clearly true for the starting case $\dim \mathscr{B} = 0$ because $H = 1$ in this case.

\subsection{Algebraic family associated with $\mathscr{B}$}\label{SubsectionAlgFamilyAssoWithB}
Mok's idea to prove the Ax type transcendence results is to use algebraic foliations. In our situation, we wish to construct a family $\mathscr{F}'$ of varieties in $\cX^+$ associated with $\mathscr{Z}$, such that $\mathscr{F}'$ is $\Gamma_0$-invariant for a suitable subgroup $\Gamma_0$ of $\Gamma$. This construction can be realized, for example, by using Hilbert schemes. Then $\bu(\mathscr{F}')$ is a foliation on $M$. Next we wish to improve the algebraicity of $\bu(\mathscr{F}')$ to make it into an algebraic subvariety. We present this process for our situation in this subsection. 
We point out that it is \cite[$\mathsection$3]{MokAx-Schanuel-for} adapted to mixed Shimura varieties of Kuga type with some slight modifications.

Let $Y = \bu(\tilde{Z})^{\mathrm{Zar}} = \mathrm{pr}_M(\mathscr{Z})^{\mathrm{Zar}}$. Let $\cX^\vee$ be as in Proposition~\ref{PropositionRealizationOfX}, then $\cX^+$ is open (in the usual topology) semi-algebraic in $\cX^\vee$. Let $\bH$ be the Hilbert scheme of all subvarieties of $\cX^\vee \times Y$ with the same Hilbert polynomial as $\mathscr{B}$, and let $\mathbscr{B} \rightarrow \bH$ be the (modified) universal family, namely
\[
\mathbscr{B} = \{(\tilde{x},m,[W])\in \cX^+ \times Y \times \bH : (\tilde{x},m) \in W \} \hookrightarrow (\cX^+ \times Y) \times \bH \subseteq (\cX^+ \times M) \times \bH
\]
where $[W]$ means the point of $\bH$ representing $W$. 
It is known that $\bH$ is proper. Denote by
\[
\psi \colon \mathbscr{B} \rightarrow \cX^+ \times M
\]
the natural projection. Then $\psi$ is proper since $\bH$ is proper.

The action of $\Gamma$ on $\cX^+ \times M$ induces an action of $\Gamma$ on $\mathbscr{B}$ by
\begin{equation}\label{EqActionGammaOnB}
\gamma(\tilde{x}, m, [W]) = (\gamma\tilde{x}, m, [\gamma W]).
\end{equation}
It is clear that $\psi$ is $\Gamma$-equivariant.

Recall that $\Delta = \mathrm{graph}(\cX^+ \rightarrow M) \subseteq \cX^+ \times M$. Define $\mathbscr{Z} \subseteq \mathbscr{B} \cap (\Delta \times \bH)$ to be
\begin{equation}\label{EqFiberDimIntersectionDim}
\mathbscr{Z} = \{(\tilde{\delta},[W]) \in \Delta \times \bH : \tilde{\delta} \in \Delta \cap W, ~\dim_{\tilde{\delta}} (\Delta \cap W)\ge \dim \mathscr{Z}\}. 
\end{equation}
Then $\mathbscr{Z}$ is a closed complex analytic subset of $\mathbscr{B}$. Hence $\psi(\mathbscr{Z})$ is closed complex analytic in $\cX^+ \times M$ since $\psi$ is proper.

Note that $\Delta$ is $\Gamma$-invariant. So $\mathbscr{Z}$ is $\Gamma$-invariant for the action \eqref{EqActionGammaOnB}. Hence we can define $\Gamma \backslash \mathbscr{Z}$, which is naturally a complex analytic variety. Moreover $\psi(\mathbscr{Z})$ is also $\Gamma$-invariant since $\psi$ is $\Gamma$-equivariant. We thus obtain the following commutative diagram:
\begin{equation}\label{EqBigDiagramForNormality}
\xymatrix{
\mathbscr{Z} \ar[r]^-{\psi} \ar[d] & \psi(\mathbscr{Z}) \ar@{}[r]|{\subseteq} \ar[d] & \cX^+ \times M \ar[d]^{(\bu,\mathrm{id})} \\
 \Gamma \backslash \mathbscr{Z} \ar[r]^-{\bar{\psi}} & \Gamma \backslash \psi(\mathbscr{Z}) \ar@{}[r]|{\subseteq} & M  \times M.
}
\end{equation}
Then $\bar{\psi}$ is proper since $\psi$ is proper.

We are now ready to prove the following result.

\begin{prop}\label{PropAlgebraicityOfTheAlgebraicFoliation}
The subset $(\bu,\mathrm{id})(\psi(\mathbscr{Z})) = \bar{\psi}(\Gamma \backslash \mathbscr{Z})$ is a closed algebraic subvariety of $M  \times M$.
\end{prop}
\begin{proof}
We claim that $\psi(\mathbscr{Z}) \cap (\mathfrak{F} \times M)$ is definable in $\R_{\mathrm{an},\exp}$. Since
\[
\psi(\mathbscr{Z}) \cap (\mathfrak{F} \times M) = \psi(\mathbscr{Z} \cap (\mathfrak{F} \times M \times \bH)),
\]
it suffices to prove that $\mathbscr{Z} \cap (\mathfrak{F} \times M \times \bH)$ is definable. But
\begin{align*}
\mathbscr{Z} \cap (\mathfrak{F} \times M \times \bH) = \{(\tilde{x},m,[W]) \in \mathfrak{F} \times M \times \bH :  & ~ (\tilde{x},m) \in W,~\tilde{x} \in (\bu|_{\mathfrak{F}})^{-1}(m), \\ 
& ~\dim_{(\tilde{x},m)} (\Delta \cap (\mathfrak{F} \times M) \cap W)\ge \dim \mathscr{Z} \}.
\end{align*}
Hence $\mathbscr{Z} \cap (\mathfrak{F} \times M \times \bH)$ is definable since $\bu|_{\mathfrak{F}}$ is definable.

Now since $\psi(\mathbscr{Z})$ is $\Gamma$-invariant, we have that
\[
(\bu,\mathrm{id})(\psi(\mathbscr{Z})) = (\bu,\mathrm{id})\left( \psi(\mathbscr{Z}) \cap (\mathfrak{F} \times M) \right)
\]
is closed complex analytic and definable in $M  \times M$. Hence $(\bu,\mathrm{id})(\psi(\mathbscr{Z}))$ is closed algebraic by Peterzil-Starchenko's o-minimal Chow \cite[Theorem~4.5]{PeterzilComplex-analyti}; see also \cite[Theorem~2.2]{MokAx-Schanuel-for}.
\end{proof}

\begin{cor}\label{CorAlgebraicityOfTheAlgebraicFoliation}
We have $\bar{\psi}(\Gamma \backslash \mathbscr{Z}) = Y$ with $Y$ viewed as a subvariety of $M  \times M$ via the diagonal embedding.
\end{cor}
\begin{proof}
First, observe that $\psi(\mathbscr{Z}) \subseteq \Delta$. Hence $\bar{\psi}(\Gamma \backslash \mathbscr{Z}) = (\bu,\mathrm{id})(\psi(\mathbscr{Z})) \subseteq M$ with $M$ viewed as a subvariety of $M \times M$ via the diagonal embedding. Next, we have $\psi(\mathbscr{Z}) \subseteq \psi(\mathbscr{B}) \subseteq \cX^+ \times Y$ by definition of $\mathbscr{B}$. Hence $\bar{\psi}(\Gamma \backslash \mathbscr{Z}) = (\bu,\mathrm{id})(\psi(\mathbscr{Z})) \subseteq \mathrm{pr}_2^{-1}(Y) \subseteq M \times M$ for the projection to the second factor $\mathrm{pr}_2 \colon M \times M \rightarrow M$. So $\bar{\psi}(\Gamma \backslash \mathbscr{Z}) \subseteq M \cap \mathrm{pr}_2^{-1}(Y) = Y$ with $Y$ viewed as a subvariety of $M \times M$ via the diagonal embedding.

By assumption we have $\mathscr{Z} \subseteq \psi(\mathbscr{Z})$. Hence $(\bu,\mathrm{id})(\mathscr{Z}) \subseteq (\bu,\mathrm{id})(\psi(\mathbscr{Z})) = \bar{\psi}(\Gamma \backslash \mathbscr{Z})$. Applying the projection $\mathrm{pr}_1 \colon M \times M \rightarrow M$ to the first factor, we get
\begin{equation}\label{EqTildeZT0}
\bu(\tilde{Z}) = \bu(\mathrm{pr}_{\cX^+}(\mathscr{Z})) \subseteq \mathrm{pr}_1(\bar{\psi}(\Gamma \backslash \mathbscr{Z})).
\end{equation}
Taking Zariski closures of both sides, we get $Y \subseteq \mathrm{pr}_1(\bar{\psi}(\Gamma \backslash \mathbscr{Z}))$. But $\bar{\psi}(\Gamma \backslash \mathbscr{Z}) \subseteq Y$ with $Y$ viewed as a subvariety of $M \times M$ via the diagonal embedding. We are done.
\end{proof}

\subsection{The $\Q$-stabilizer and first steps towards the normality}\label{SubsectionFirstStepTowardsNormality}
In this subsection we relate $Y$ to the $\Q$-stabilizer of $\mathscr{B}$ via monodromy. We point out that it is \cite[$\mathsection$3]{MokAx-Schanuel-for} adapted to the universal abelian variety. However we need to carefully distinguish the Mumford-Tate group and the monodromy group in the mixed case, since the underlying group is no longer reductive (so that we do not have a decomposition $\cX^+ = \cX_1^+ \times \cX_2^+$ as in the pure case, and are thus not able to replace $\cX^+$ by an $N(\R)^+$-orbit).


Let $\mathbscr{Z}_0$ be the complex analytic irreducible component of $\mathbscr{Z}$ which contains $\mathscr{Z} \times [\mathscr{B}]$. By Corollary~\ref{CorAlgebraicityOfTheAlgebraicFoliation},  $\bar{\psi}(\Gamma\backslash\mathbscr{Z}_0) = Y$ since $Y$ is irreducible. Consider the following map
\begin{equation}\label{EqLongMapFundamentalGroups}
\pi_1(\Gamma\backslash\mathbscr{Z}_0) \xrightarrow{\bar{\psi}_*} \pi_1(Y) \rightarrow \Gamma.
\end{equation}
Denote by $\Gamma_0$ the image of this map, and by $N$ the identity component of the $\Q$-Zariski closure (in $P$) of $\Gamma_0$. Then it is clear that $\Gamma_0 \mathbscr{Z}_0 \subseteq \mathbscr{Z}_0$ for the action $\Gamma$ on $\mathbscr{B}$ defined by \eqref{EqActionGammaOnB}.
\begin{lemma}\label{LemmaNnormalinP}
We have  $N \lhd P$. Moreover $\tilde{Z}^{\biZar} = N(\R)^+\tilde{z}$ for any $\tilde{z} \in \tilde{Z}$.
\end{lemma}
\begin{proof}
Recall our assumption that $(P,\cX^+)$ is the smallest connected mixed Shimura subdatum of Kuga type such that $\tilde{Z} \subseteq \cX^+$. Therefore $M$ is the smallest connected mixed Shimura subvariety of Kuga type which contains $Y = \bu(\tilde{Z})^{\mathrm{Zar}}$.

Since $\bar{\psi}$ is proper, the image of $\bar{\psi}_*$ has finite index in $\pi_1(Y)$. 
Hence $N$ equals the connected algebraic monodromy group of $Y$. By the last paragraph, we have $N \lhd P$ by Andr\'{e} \cite[$\mathsection$5, Theorem~1]{AndreMumford-Tate-gr}; see \cite[$\mathsection$3.3]{GaoTowards-the-And}.

Next observe that $\bu(\tilde{Z})^{\biZar} = Y^{\biZar}$. So $\tilde{Z}^{\biZar} = N(\R)^+\tilde{z}$ (for any $\tilde{z} \in \tilde{Z}$) by \cite[Theorem~8.1]{GaoTowards-the-And}. Thus we finish the proof.
\end{proof}

Denote by $\theta \colon \mathbscr{B} \rightarrow \bH$ the natural projection. Let $\mathbscr{F} = \theta^{-1}(\theta(\mathbscr{Z}_0)) = \{(\tilde{x},m,[W]) : [W] \in\theta(\mathbscr{Z}_0), ~ (\tilde{x},m) \in W\}$, then $\mathbscr{F} \subseteq \mathbscr{B}$ is the family of algebraic varieties parametrized by $\theta(\mathbscr{Z}_0) \subseteq \bH$, with the fiber over each $[W] \in \theta(\mathbscr{Z}_0)$ being $W$. Then we have
\[
\Gamma_0 \mathbscr{F} \subseteq \mathbscr{F}
\]
for the action $\Gamma$ on $\mathbscr{B}$ defined by \eqref{EqActionGammaOnB}; see below \eqref{EqLongMapFundamentalGroups}.

For any $[W] \in \theta(\mathbscr{Z}_0)$ and any complex analytic irreducible component $\mathscr{Z}'$ of its fiber $(\mathbscr{Z}_0)_{[W]}$, we have $\bu(\mathrm{pr}_{\cX^+}(\mathscr{Z}'))^{\mathrm{Zar}} \subseteq Y$ as $\mathscr{Z}' \subseteq \psi(\mathbscr{Z}_0)$. 
Thus $\mathrm{pr}_{\cX^+}(\mathscr{Z}')^{\biZar} \subseteq \tilde{Z}^{\biZar}$ as $Y^{\biZar} = \bu(\tilde{Z})^{\biZar}$. Moreover for a very general $[W]\in \theta(\mathbscr{Z}_0)$, \textit{i.e.} outside a countable union of proper closed subvarieties, there exists a component $\mathscr{Z}'$ of $(\mathbscr{Z}_0)_{[W]}$ such that $\mathrm{pr}_{\cX^+}(\mathscr{Z}')^{\biZar} = \tilde{Z}^{\biZar}$.

Denote by
\[
\Gamma_{\mathbscr{F}} = \{\gamma \in \Gamma : \gamma W \subseteq W,~\forall [W] \in \theta(\mathbscr{Z}_0)\}.
\]
Then for a very general $[W] \in \theta(\mathbscr{Z}_0)$ (but may be different from above), we have
\begin{equation}\label{EqStabOfGeneralFiber}
\mathrm{Stab}_{\Gamma}(W) = \Gamma_{\mathbscr{F}}.
\end{equation}
Indeed, let $\bH_{\gamma} := \{[W] \in \theta(\mathbscr{Z}_0) : \gamma W \subseteq W\}$ for each $\gamma \in \Gamma$. Then for any $\gamma \not \in \Gamma_{\mathbscr{F}}$, we have that $\bH_{\gamma}$ is a proper subvariety of $\theta(\mathbscr{Z}_0)$. Thus it suffices to take $[W] \in \theta(\mathbscr{Z}_0) \setminus \bigcup_{\gamma \in \Gamma \setminus \Gamma_{\mathbscr{F}}} \bH_{\gamma}$.

The two paragraphs above together imply: for a very general $[W]\in \theta(\mathbscr{Z}_0)$, we have $\mathrm{Stab}_{\Gamma}(W) = \Gamma_{\mathbscr{F}}$ and there exists a component $\mathscr{Z}'$ of $(\mathbscr{Z}_0)_{[W]}$ such that $\mathrm{pr}_{\cX^+}(\mathscr{Z}')^{\biZar} = \tilde{Z}^{\biZar}$.

It is known that $\mathscr{B}$ is a fiber of $\mathbscr{F}$. We claim that it suffices to handle the case where $\mathscr{B}$ is a very general fiber as in the previous paragraph. Indeed, take $[W] \in \theta(\mathbscr{Z}_0)$ very general. Let $\mathscr{Z}'$ be an irreducible component of $(\mathbscr{Z}_0)_{[W]}$ with $\bu(\mathrm{pr}_{\cX^+}(\mathscr{Z}'))^{\mathrm{biZar}} = \tilde{Z}^{\biZar}$. We have $\dim \mathscr{Z}' \ge \dim \mathscr{Z}$ by definition of $\mathbscr{Z}$, and $\dim W = \dim \mathscr{B}$ by definition of the Hilbert scheme $\bH$. Denote by $\mathscr{B}' = (\mathscr{Z}')^{\Zar}$, then $\mathscr{B}' \subseteq W$. Then $\dim \mathscr{B} - \dim \mathscr{Z} \ge \dim \mathscr{B}' - \dim \mathscr{Z}'$ and $\dim\tilde{Z}^{\biZar} = \dim \mathrm{pr}_{\cX^+}(\mathscr{Z}')^{\biZar}$. So to prove the Ax-Schanuel Theorem (Theorem~\ref{ThmEquivalentFormOfAS}) for $(\mathscr{B},\mathscr{Z})$, it suffices to prove it for $(\mathscr{B}', \mathscr{Z}')$. If $\dim \mathscr{B}' < \dim W = \dim \mathscr{B}$, then Proposition~\ref{PropNormOfQStab} follows from induction hypothese and we are done. If $\dim \mathscr{B}' = \dim W$, then $\mathscr{B}' = W$ is a very general fiber of $\mathbscr{F}$ as we desire.

 By \eqref{EqStabOfGeneralFiber}, $(\Gamma_{\mathbscr{F}}^{\Zar})^{\circ}$ is the $\Q$-stabilizer of $\mathscr{B}$, which is the $H$ defined in \eqref{EqDefnQStabilizerOfB}. The ``Moreover'' part of Lemma~\ref{LemmaNnormalinP} implies $\mathscr{B} = N(\R)^+\tilde{z} \times \bu(N(\R)^+\tilde{z})$, and hence $H$ is a subgroup of $N$.


\begin{lemma}\label{LemmaHnormalinN}
We have $H \lhd N$.
\end{lemma}
\begin{proof} Since $\Gamma_0 \mathbscr{F} \subseteq \mathbscr{F}$, every $\gamma \in \Gamma_0$ sends a very general fiber $W$ of $\mathbscr{F}$ to a very general fiber $W'$ of $\mathbscr{F}$. As $\mathrm{Stab}_{\Gamma}(W') = \mathrm{Stab}_{\Gamma}(\gamma W) = \gamma \mathrm{Stab}_{\Gamma}(W) \gamma^{-1}$, by \eqref{EqStabOfGeneralFiber} we get $\Gamma_{\mathbscr{F}} = \gamma \Gamma_{\mathbscr{F}} \gamma^{-1}$ for all $\gamma \in \Gamma_0$. Hence the conclusion follows by taking Zariski closures.
\end{proof}


\subsection{Normality of the $\Q$-stabilizer}\label{SubsectionLastStepTowardsNormality}
The argument in this subsection is new compared to the proof of the pure Ax-Schanuel theorem.


\begin{proof}[Proof of Proposition~\ref{PropNormOfQStab}]
Recall the definable set $\Theta$ defined in \eqref{EquationXiSetsForFiberOfMSV}
\[
\Theta=\{p\in P(\R): \dim(p^{-1}\mathscr{B}\cap  (\mathfrak{F} \times M) \cap \Delta)=\dim \mathscr{Z}\},
\]
where $\mathfrak{F}$ is the fundamental set for $\bu \colon \cX^+ \rightarrow M$ defined in $\mathsection$\ref{SubsectionFundamentalSet}. 
By Remark~\ref{RmkAssumptionASEnlarge}, we may and do assume that every positive dimensional semi-algebraic block in $\Theta$ is contained in a left coset of $\stab_{P(\R)}(\mathscr{B})$.

We have $H(\R)^+ \mathscr{B} \subseteq \mathscr{B}$ since $\mathscr{B}$ is algebraic in $\cX^+ \times M$ and the action of $P(\R)^+$ on $\cX^+ \times M$ is semi-algebraic.

Let $V_H = V \cap H$ and $V_N = V \cap N$. Let $G_H = H/V_H$, $G_N = N/V_N$ and $G = P/V$. Then $G_H \lhd G_N \lhd G$ by Lemma~\ref{LemmaNnormalinP} and Lemma~\ref{LemmaHnormalinN}. Hence $G_H \lhd G$ since any reductive group over $\Q$, in particular $G$, is an almost direct product of its center and almost simple normal subgroups.

Now we are left to prove the following two facts.
\begin{enumerate}
\item[(i)] $V_H$ is a $G$-submodule of $V$.
\item[(ii)] $G_H$ acts trivially on $V/V_H$.
\end{enumerate}



We start with (ii). By Lemma~\ref{LemmaHnormalinN}, we have $H\lhd N$. It follows that $G_H$ acts trivially on $V_N/V_H$. Hence it suffices to prove that $G_H$ acts trivially on $V/V_N$. But Lemma~\ref{LemmaNnormalinP} asserts that $N \lhd P$, and hence $G_N$ acts trivially on $V/V_N$. In particular $G_H$ acts trivially on $V/V_N$. Thus we have established (ii).

Now we prove (i). The ``Moreover'' part of Lemma~\ref{LemmaNnormalinP} implies that $\mathscr{B} \subseteq \tilde{F} \times F$, where $\tilde{F} = N(\R)^+\tilde{z}$ and $F = \bu(\tilde{F})$. Consider $\tilde{X} = \mathrm{pr}_{\cX^+}(\mathscr{B})$, then $\tilde{X} \subseteq \tilde{F}$. Applying the natural projection $\tilde{\pi} \colon (P,\cX^+) \rightarrow (G,\cX_G^+)$, we get $\tilde{X}_G \subseteq G_N(\R)^+\tilde{z}_G$. Here we denote by $\tilde{X}_G = \tilde{\pi}(\tilde{X})$ and $\tilde{z}_G = \tilde{\pi}(\tilde{z})$.

For convenience of readers, the rest of the proof of (i) is divided into 3 steps.

\noindent\boxed{\text{Step~I}} Any $\tilde{x}_G \in \cX_G^+$ gives a morphism $\C^\times = \S(\R) \rightarrow G(\R)$. If we endow $V(\R)$ with the complex structure determined by $V(\R) \cong \tilde{\pi}^{-1}(\tilde{x}_G)$, then a real subspace $V'$ of $V_{\R}$ is complex if and only if the following condition holds: $\tilde{x}_G(\sqrt{-1}) \cdot V' \subseteq V'$. 
Denote for simplicity by $J_{\tilde{x}_G} = \tilde{x}_G(\sqrt{-1})$. Thus under this complex structure on $V(\R)$, we have
\begin{itemize}
\item the smallest complex subspace of $V(\R)$ containing $V_H(\R)$ is $V_H(\R) + J_{\tilde{x}_G} V_H(\R)$.
\item the largest complex subspace of $V(\R)$ contained in $V_H(\R)$ is $V_H(\R) \bigcap J_{\tilde{x}_G} V_H(\R)$.
\end{itemize}

Now each point in $\tilde{X}_G$ is of the form $g\tilde{z}_G$ for some $g \in G_N(\R)^+$. 

We claim that $J_{\tilde{x}_G} V_H(\R)$ is independent of $\tilde{x}_G$ when $\tilde{x}_G$ varies in $G_N(\R)^+\tilde{z}_G$. For any $g\in G_N(\R)^+$, we have $J_{\tilde{x}_G}^{-1} g J_{\tilde{x}_G} \in G_N(\R)^+$ since $G_N \lhd G$. Hence $g J_{\tilde{x}_G} = J_{\tilde{x}_G} g'$ for some $g' \in G_N(\R)^+$. Note that $J_{g\tilde{x}_G} = gJ_{\tilde{x}_G} g^{-1}$. Thus
\[
J_{g\tilde{x}_G} V_H(\R) = gJ_{\tilde{x}_G} g^{-1} \cdot V_H(\R) \subseteq g J_{\tilde{x}_G} \cdot V_H(\R) = J_{\tilde{x}_G} g' \cdot V_H(\R) \subseteq J_{\tilde{x}_G} \cdot V_H(\R).
\]
Here we used the fact that $V_H$ is stable under $G_N$ (since $H \lhd N$). This proves the claim.

In particular, $J_{\tilde{x}_G} V_H(\R)$ is independent of $\tilde{x}_G$ for $\tilde{x}_G \in \tilde{X}_G$.



\noindent\boxed{\text{Step~II}} Next for any $\gamma_G \in \tilde{\pi}(\Theta) \cap \Gamma_G \supseteq \{\gamma_G \in \Gamma_G : \gamma_G \mathfrak{F}_G \cap \tilde{Z}_G \not= \emptyset\}$,\footnote{It is not hard to show this inclusion by using $\Theta \cap \Gamma \supseteq \{\gamma \in \Gamma: \gamma\mathfrak{F} \cap \tilde{Z} \not= \emptyset\}$. See \cite[proof of Lemma~10.2]{GaoTowards-the-And}.}  let $\tilde{Z}|_{\gamma_G \mathfrak{F}_G}^+$ be a complex analytic irreducible component of $\tilde{Z}|_{\gamma_G \mathfrak{F}_G} := \tilde{Z} \cap \tilde{\pi}^{-1}(\gamma_G \mathfrak{F}_G)$, and define the following definable set
\[
\Theta' = \{v \in V(\R) : \dim\big( (-v,\gamma_G^{-1}) \mathscr{B} \cap (\mathfrak{F} \times M) \cap \Delta \big) = \dim \mathscr{Z} \} \subseteq V(\R).
\]
Then $\Theta' \cap \Gamma \supseteq \{\gamma_V \in \Gamma_V : (\gamma_V,\gamma_G)\mathfrak{F} \cap \tilde{Z}|_{\gamma_G \mathfrak{F}_G}^+ \not= \emptyset\}$, and $(\Theta', \gamma_G^{-1}) \subseteq \Theta$.
Denote by $\Gamma_{V_H} = V_H(\Q) \cap \Gamma_V$.

Suppose $\{\gamma_V \in \Gamma_V : (\gamma_V,\gamma_G)\mathfrak{F} \cap \tilde{Z}|_{\gamma_G \mathfrak{F}_G}^+ \not= \emptyset\}$ is not contained in a finite union of $\Gamma_{V_H}$-cosets. Identify $\cX^+ \cong V(\R) \times \cX_G^+$, then $\mathfrak{F} = (-k,k)^{2g} \times \mathfrak{F}_G$ for some $k \ge 1$, where $g = \dim V$. Hence
\[
\{\gamma_V \in \Gamma_V : (\gamma_V,\gamma_G)\mathfrak{F} \cap \tilde{Z}|_{\gamma_G \mathfrak{F}_G}^+ \not= \emptyset\} = \{\gamma_V \in \Gamma_V :  \left( (\gamma_V + \gamma_G(-k,k)^{2g}) \times \gamma_G \mathfrak{F}_G \right) \cap \tilde{Z}|_{\gamma_G \mathfrak{F}_G}^+ \not= \emptyset\}.
\]
Let $p_V \colon \cX^+ \cong V(\R) \times \cX^+_G \rightarrow V(\R)$, then $p_V((\gamma_V + \gamma_G(-k,k)^{2g}) \times \gamma_G \mathfrak{F}_G)$ is contained in a Euclidean ball of radius at most $kH(\gamma_G)$. But $\tilde{Z}|_{\gamma_G \mathfrak{F}_G}^+$ is connected by choice. 
Thus if the set above is not contained in a finite union of $\Gamma_{V_H}$-cosets, then we get for free that it contains $\ge T$ elements in $\Gamma_V \setminus \Gamma_{V_H}$ of height at most $H(\gamma_G)T$ (for all $T\gg 0$). \footnote{The crucial point is that the group $V(\R)$ is Euclidean. See \cite[pp.~218-219]{TsimermanAx-Schanuel-and}.} 
Hence by Pila-Wilkie \cite[Theorem~3.6]{PilaO-minimality-an}, there exist two constants $c,\epsilon > 0$ with the following property: for each $T \gg 0$, $\Theta'$ contains a  semi-algebraic block $B'$ which is not in any coset of $V_H(\R)$ and which contains $\ge c T^{\epsilon}$ elements in $\Gamma_V$ outside $\Gamma_{V_H}$ of height at most $T$. Recall our assumption that every positive dimensional semi-algebraic block in $\Theta$ is contained in a left coset of $\stab_{P(\R)}(\mathscr{B})$. In particular $(B', \gamma_G^{-1}) \subseteq (\gamma_V' , \gamma_G^{-1}) \cdot \stab_{P(\R)}(\mathscr{B})$ for some $\gamma_V' \in B' \cap \Gamma_V$. Hence $(-\gamma_V' + \gamma_G B', 1) \subseteq \stab_{P(\R)}(\mathscr{B})$; the argument is similar to the end of Proposition~\ref{PropBignessOfQStabilizer} (right above Remark~\ref{RmkAssumptionASEnlarge}). But then
\[
\left( (-\gamma_V' + \gamma_G B') \cap \Gamma_V, 1 \right)\subseteq \stab_{P(\R)}(\mathscr{B}) \cap \Gamma \subseteq H(\Q).
\]
By letting $T \rightarrow \infty$ and varying $B'$ accordingly, we see that this inclusion cannot hold since each $B' \subseteq V(\R)$ thus obtained is not contained in any coset of $V_H(\R)$.

Thus $\{\gamma_V \in \Gamma_V : (\gamma_V,\gamma_G)\mathfrak{F} \cap \tilde{Z}|_{\gamma_G \mathfrak{F}_G}^+ \not= \emptyset\}$ is contained in a finite union of $\Gamma_{V_H}$-cosets. Hence for the map $\bar{\bu} \colon \cX^+ \rightarrow \Gamma_V \backslash \cX^+$,\footnote{The complex space $\Gamma_V \backslash \cX^+$ is a family of abelian varieties. In fact it is the pullback of the abelian scheme $M \rightarrow M_G$ under the uniformization $\cX_G^+ \rightarrow M_G$.} we have that $\bar{\bu}(\tilde{Z}|_{\gamma_G \mathfrak{F}_G}^+)$ is closed in $(\Gamma_V \backslash \cX^+)|_{\gamma_G \mathfrak{F}_G}$ in the usual topology.

In the following paragraph we consider the usual topology. As $\gamma_G\mathfrak{F}_G$ is an open subset of $\cX^+_G$ for each $\gamma_G \in \Gamma_G$, each $\tilde{Z}|_{\gamma_G\mathfrak{F}_G}^+$ is an open subset of $\tilde{Z}$. In particular there exists a $\tilde{Z}|_{\gamma_G\mathfrak{F}_G}^+$ as above such that $\dim \tilde{Z}|_{\gamma_G\mathfrak{F}_G}^+ = \dim \tilde{Z}$. Consider the closure $W$ of $\bar{\bu}(\tilde{Z}|_{\gamma_G\mathfrak{F}_G}^+)$ in $\Gamma_V \backslash \cX^+$. By the last paragraph, we have $W^{\circ} \subseteq \bar{\bu}(\tilde{Z}|_{\gamma_G\mathfrak{F}_G}^+) \subseteq W$ where $W^{\circ}$ is the interior of $W$. Now let $\tilde{Z}'$ be the complex analytic irreducible component of $\bar{\bu}^{-1}(W)$ that contains $\tilde{Z}|_{\gamma_G\mathfrak{F}_G}^+$. Then we have $\tilde{Z}' \subseteq \tilde{Z}$ and $\bar{\bu}(\tilde{Z}') = W$ is closed in $\Gamma_V\backslash \cX^+$.

Denote by $\mathscr{Z}' := \mathrm{graph}(\tilde{Z}' \rightarrow \bu(\tilde{Z}'))$. Then $\dim \mathscr{Z}' = \dim \mathscr{Z}$ because $\dim \tilde{Z}' = \dim \tilde{Z}$. By analytic continuation, we then have $(\mathscr{Z}')^{\Zar} = \mathscr{Z}^{\Zar} = \mathscr{B}$ and $(\tilde{Z}')^{\Zar} = \tilde{Z}^{\Zar} = \tilde{X}$.


\noindent\boxed{\text{Step~III}} 
Now let
\[
\Gamma' = \im \left( \pi_1(\bar{\bu}(\tilde{Z}')) \rightarrow \pi_1(\Gamma_V \backslash \cX^+) = \Gamma_V \right) \subseteq \Gamma_V.
\]
Then $\Gamma'$ stabilizes $\tilde{Z}'$ for the action of $\Gamma_V$ on $\cX^+$. Recall that $P(\R)^+$ acts on $\cX^+ \times M$ via its action on the first factor. So $\Gamma'$ stabilizes $\mathscr{Z}'$. Hence $(\Gamma')^{\Zar}(\R) \subset V(\R)$ stabilizes $\mathscr{B} = (\mathscr{Z}')^{\Zar}$. On the other hand $H$ is defined to be the $\Q$-stabilizer of $\mathscr{B}$. So $(\Gamma')^{\Zar} \subseteq H \cap V = V_H$.

Let us take a closer look at this. The identification $\pi_1(\Gamma_V \backslash \cX^+) = \Gamma_V$ is realized via \eqref{EqiDelta}
\[
i_{\cX_G^+}^{-1} \colon \cX^+ \cong V(\R) \times \cX_G^+
\]
and hence $\Gamma_V \backslash \cX^+ \cong (\Gamma_V\backslash V(\R)) \times \cX_G^+$. We have $\tilde{Z}' \subseteq \tilde{X} = (\tilde{Z}')^{\Zar} \subseteq \tilde{F} = N(\R)^+\tilde{z}$. Since $V_H \lhd N$, we can take the quotient of $\tilde{F} = N(\R)^+\tilde{z}$ by $V_H(\R)$ and get a real manifold. Call this quotient $q$. Denote by $\Gamma_{V_N} = \Gamma_V \cap N(\Q)$ and $\Gamma_{V_H} = \Gamma_V \cap H(\Q)$. Then we obtain a commutative diagram
\[
\xymatrix{
V_N(\R) \times \tilde{F}_G \cong \tilde{F} \ar[r]^-{q} \ar[d]_{\bar{\bu}} & q(\tilde{F}) \cong (V_N/V_H)(\R) \times \tilde{F}_G \ar[d]  \\
(\Gamma_{V_N}\backslash V_N(\R)) \times \tilde{F}_G \cong \bar{\bu}(\tilde{F}) \ar[r]^-{[q]} &  \big( (\Gamma_{V_N}/\Gamma_{V_H})\backslash (V_N/V_H)(\R) \big) \times \tilde{F}_G.
}
\]
Here all the isomorphisms in the diagram are compatible with the $i_{\cX_G^+}^{-1}$ above. For now on, by abuse of notation we no longer write $i_{\cX_G^+}^{-1}(\cdot)$.
 
Since $\Gamma' \subseteq V_H(\Q)$, we have that $[q](\bar{\bu}(\tilde{Z}'))$ is contained in $\{t\} \times \tilde{F}_G$ for some $t \in (\Gamma_{V_N}/\Gamma_{V_H})\backslash (V_N/V_H)(\R)$. But then $q(\tilde{Z}') \subseteq \{v'\} \times \tilde{F}_G$ for some $v' \in (V_N/V_H)(\R)$. Thus there exists $v \in V(\R)$ such that 
\begin{equation}\label{EqInclusionTildeZFamily}
\tilde{Z}' - (\{v\} \times \tilde{Z}'_G) \subseteq V_H(\R) \times \tilde{F}_G.
\end{equation}
But $i_{\cX_G^+}(\{v\} \times \tilde{Z}'_G)$ is complex analytic in $\cX^+$ by property (ii) of $i_{\cX_G^+}$ below \eqref{EqiDelta}. So the left hand side of \eqref{EqInclusionTildeZFamily} is complex analytic. 

We have proven in Step~I that over each $\tilde{x}_G \in \tilde{F}_G$, the largest complex subspace of $V(\R)$ contained in $V_H(\R)$ is $V_H(\R) \bigcap J_{\tilde{x}_G} V_H(\R)$ and $J_{\tilde{x}_G} V_H(\R)$ is independent of the choice of $\tilde{x}_G$. Hence \eqref{EqInclusionTildeZFamily} implies
\[
\tilde{Z}' - (\{v\} \times \tilde{Z}'_G) \subseteq (V_H(\R) \bigcap J_{\tilde{x}_G} V_H(\R)) \times \tilde{F}_G
\]
since the left hand side is complex analytic. Taking the Zariski closures, we get
\[
\tilde{X} - (\{v\} \times \tilde{X}_G) \subseteq (V_H(\R) \bigcap J_{\tilde{x}_G} V_H(\R)) \times \tilde{F}_G.\footnote{We use the fact that any subset of $\cX^+$ is algebraic if and only if it is complex analytic and semi-algebraic.}
\]
But $V_H(\R)$ stabilizes $\tilde{X}$, so $V_H(\R) = V_H(\R) \bigcap J_{\tilde{x}_G} V_H(\R)$. But then $V_H(\R)$ is complex for the complex structure of $V(\R) \cong \cX^+_{\tilde{x}_G}$ for any $\tilde{x}_G \in \tilde{F}_G$. Since $\tilde{Z}$ is assumed to be Hodge generic in $(P,\cX^+)$, we may take $\tilde{x}_G$ such that $\mathrm{MT}(\tilde{x}_G) = G$. Then we see that $V_H$ is a $G$-module.
\end{proof}

\section{End of proof}\label{SectionEndOfProof}
Use the notation of $\mathsection$\ref{SectionQStabBig}. We finish the proof of Theorem~\ref{ThmEquivalentFormOfAS}. Let $\mathscr{B}$ and $\mathscr{Z} = \mathrm{graph}(\tilde{Z} \rightarrow \bu(\tilde{Z}))$ be as in the theorem. Denote by $\tilde{F} = \tilde{Z}^{\biZar}$. Then $\tilde{F}$ is weakly special by \cite[Theorem~8.1]{GaoTowards-the-And}, and hence $\tilde{F} = N(\R)^+\tilde{x}$ for some $N \lhd P$ and some $\tilde{x} \in \cX^+$. See $\mathsection$\ref{SubsectionWeaklySpecial}.

By Proposition~\ref{PropBignessOfQStabilizer}, it suffies to consider the case where $\dim H > 0$. By Proposition~\ref{PropNormOfQStab}, we may and do assume $H \lhd P$. Note that $H$ is a subgroup of $N$ by Lemma~\ref{LemmaNnormalinP} and Lemma~\ref{LemmaHnormalinN}.

Consider the quotient connected mixed Shimura datum of Kuga type $(P,\cX^+)/H$, which we denote by $(P',\cX^{\prime+})$. Use $M'$ to denote the corresponding connected mixed Shimura variety of Kuga type. We have the following commutative diagram:
\[
\xymatrix{
(P,\cX^+) \ar[r]^-{\tilde{\rho}} \ar[d]_{\bu} & (P',\cX^{\prime+}) \ar[d]^{\bu'} \\
M \ar[r]^{\rho} & M'.
}
\]
Let $\mathscr{B}' = (\tilde{\rho},\rho)(\mathscr{B})$, $\mathscr{Z}' = (\tilde{\rho},\rho)(\mathscr{Z})$ and $\tilde{Z}' = \tilde{\rho}(\tilde{Z})$. Then $\tilde{\rho}(\tilde{F}) = (\tilde{Z}')^{\mathrm{biZar}}$. Apply Proposition~\ref{PropBignessOfQStabilizer} to $\mathscr{B}'$ and $\mathscr{Z}'$. We have either
\begin{equation}\label{EqASAfterQuotient}
\dim \mathscr{B}' - \dim \mathscr{Z}' \ge \dim \tilde{\rho}(\tilde{F}),
\end{equation}
or the $\Q$-stabilizer of $\mathscr{B}'$, which we call $H'$, has positive dimension. But in the second case, $\tilde{\rho}^{-1}(H')(\R)^+$ stabilizes $\mathscr{B}$ and is larger than $H$, contradicting the maximality of $H$. So we are in the first case, namely \eqref{EqASAfterQuotient} holds.

For the quotient $(\tilde{\rho},\rho) \colon \cX^+ \times M \rightarrow \cX^{\prime+} \times M'$, 
any fiber of $\mathscr{B} \rightarrow \mathscr{B}'$ is of the form
\begin{equation}\label{EqFiberOfBBprime}
H(\R)^+\tilde{x} \times Y
\end{equation}
for some $\tilde{x} \in \mathrm{pr}_{\cX^+}(\mathscr{B})$ and some algebraic subvariety $Y$ of $M$ (with $\rho(Y)$ being a point). Hence any fiber of $\mathscr{Z}' \rightarrow \mathscr{Z}$ is of the form
\begin{equation}\label{EqFiberOfZZprime}
\mathrm{graph}(\tilde{Y} \rightarrow Y)
\end{equation}
where $\tilde{Y}$ is a complex analytic irreducible component of $\bu^{-1}(Y)$. By generic flatness we may choose a point $z' \in \mathscr{Z}'$ such that
\[
\dim \mathscr{Z} - \dim \mathscr{Z}' = \dim \mathscr{Z}_{z'} \quad \text{ and } \quad \dim \mathscr{B} - \dim \mathscr{B}' = \dim \mathscr{B}_{z'}.
\]
Thus 
\begin{align*}
\dim \mathscr{B} - \dim \mathscr{Z} & = (\dim \mathscr{B}' + \dim \mathscr{B}_{z'}) - (\dim \mathscr{Z}' + \dim \mathscr{Z}_{z'}) \\
& = (\dim \mathscr{B}' - \dim \mathscr{Z}') + (\dim (H(\R)^+\tilde{x} \times Y) - \dim Y) \quad\text{by \eqref{EqFiberOfBBprime} and \eqref{EqFiberOfZZprime}}\\
& \ge \dim \tilde{\rho}(\tilde{F}) + \dim H(\R)^+\tilde{x} \qquad \text{ by \eqref{EqASAfterQuotient}}\\
& = \dim (N/H)(\R)^+\tilde{\rho}(\tilde{x}) + \dim H(\R)^+\tilde{x} \\
& = \dim N(\R)^+\tilde{x} = \dim \tilde{F}.
\end{align*}
Hence we are done.

\section{Application to a finiteness result \textit{\`{a} la Bogomolov}}\label{SectionAppToAFinitenessResult}
The goal of this section is to prove a finiteness result \textit{\`{a} la Bogomolov}. Fix a connected mixed Shimura variety of Kuga type $M$ associated with $(P,\cX^+)$, and use $\bu \colon \cX^+ \rightarrow M$ to denote the uniformization. Fix an irreducible subvariety $Y$ of $M$.

The following definition was introduced by Habegger-Pila \cite{HabeggerO-minimality-an} to study the Zilber-Pink conjecture. Recall that  weakly special subvarieties of $M$ are precisely the bi-algebraic subvarieties; see $\mathsection$\ref{SubsectionGeomDescriptionBiAlg}.
\begin{defn}\label{DefnWeaklyOptimalSubvarieties}
\begin{enumerate}
\item[(i)] For any irreducible subvariety $Z$ of $M$, define $\delta_{\mathrm{ws}}(Z) = \dim Z^{\biZar} - \dim Z$.
\item[(ii)] A closed irreducible subvariety $Z$ of $Y$ is said to be \textbf{weakly optimal} if the following condition holds: $Z \subsetneq Z' \subseteq Y \Rightarrow \delta_{\mathrm{ws}}(Z') > \delta_{\mathrm{ws}}(Z)$, where $Z'$ is assumed to be irreducible.
\end{enumerate}
\end{defn}

\begin{thm}\label{ThmFinitenessAlaBogomolov}
There exists a finite set $\Sigma$ consisting of elements of the form $\left( (Q,\cY^+),N \right)$, where $(Q,\cY^+)$ is a connected mixed Shimura subdatum of $(P,\cX^+)$ and $N$ is a normal subgroup of $Q$ whose reductive part is semi-simple, such that the following property holds. If a closed irreducible subvariety $Z$ of $Y$ is weakly optimal, then there exists $\left( (Q,\cY^+),N \right) \in \Sigma$ such that $Z^{\biZar} = \bu(N(\R)^+\tilde{y})$ for some $\tilde{y} \in \cY^+$.
\end{thm}
Our proof of Theorem~\ref{ThmFinitenessAlaBogomolov} follows the guideline of Daw-Ren \cite[Proposition~3.3]{DawRenAppOfAS} for the corresponding result for pure Shimura varieties, and plugs in the author's previous work on extending a finiteness result for pure Shimura varieties by Ullmo \cite[Th\'{e}or\`{e}me~4.1]{UllmoQuelques-applic} to the mixed case \cite[Theorem~12.2]{GaoTowards-the-And}.

\subsection{Preliminary}
We give an equivalent statement of the weak Ax-Schanuel theorem formulated by Habegger-Pila \cite{HabeggerO-minimality-an}. We start by the following definition.
\begin{defn}
\begin{enumerate}
\item[(i)] For any complex analytic irreducible subset $\tilde{Z}$ of $\cX^+$, define $\delta_{\Zar}(\tilde{Z}) = \dim \tilde{Z}^{\Zar} - \dim \tilde{Z}$.
\item[(ii)] A closed complex analytic irreducible subset $\tilde{Z}$ of $\bu^{-1}(Y)$ is said to be \textbf{Zariski optimal} if the following condition holds: $\tilde{Z} \subsetneq \tilde{Z}' \subseteq \bu^{-1}(Y) \Rightarrow \delta_{\Zar}(\tilde{Z}') > \delta_{\Zar}(\tilde{Z})$, where $\tilde{Z}'$ is assumed to be complex analytic irreducible.
\end{enumerate}
\end{defn}
It is clear that if $\tilde{Z}$ is Zariski optimal in $\bu^{-1}(Y)$, then $\tilde{Z}$ is a complex analytic irreducible component of $\tilde{Z}^{\Zar} \cap \bu^{-1}(Y)$.

\begin{thm}\label{ThmASZariskiOptimal}
Let $\tilde{Z}$ be a complex analytic irreducible subset of $\bu^{-1}(Y)$ which is Zariski optimal. Then 
$\tilde{Z}^{\Zar}$ is weakly special.
\end{thm}
\begin{proof} This is equivalent to the weak Ax-Schanuel theorem (Theorem~\ref{ThmWASUnivAbVar}). See \cite[$\mathsection$5]{HabeggerO-minimality-an}.
\end{proof}

By Theorem~\ref{ThmReductionLemma} there exists a Shimura embedding $(P,\cX^+) \hookrightarrow (G_0,\cD^+) \times (P_{2g,D,\mathrm{a}},\cX_{2g,\mathrm{a}}^+)$ for some connected pure Shimura datum $(G_0,\cD^+)$ and some $g\times g$-diagonal matrix $D$. Let $V$ be the unipotent radical of $P$ and let $G$ be $P/V$. Denote by $\tilde{\pi} \colon (P,\cX^+) \rightarrow (G,\cX^+_G)$. Fix a Levi decomposition $P = V \rtimes G$. It induces $\cX^+ \cong V(\R) \times \cX^+_G$. All semi-direct products taken below are assumed to be compatible with this one.

Let $\cT$ be the set of pairs $(V',G')$ consisting of a symplectic (with respect to the non-degenerate alternative form $\Psi$ in \eqref{EqAltForm}) subspace of $V_{\R}$ and a connected subgroup of $G_{\R}$ which is semi-simple and has no compact factors. Let
\[
\cG := \Sp_{2g,D}(\R) \times G(\R).
\]
Then $\cG$ acts on $\cT$ by $(g_V, g) \cdot (V',G') = (g_V V', g G' g^{-1})$. Up to the action of $\cG$ on $\cT$, there exist only finitely many such pairs; see \cite[Lemma~12.3]{GaoTowards-the-And}. Fix $\Omega$ a finite set of representatives.

Finally for an element $\underline{t} = (V',G') \in \cT$ and a point $(\tilde{x}_V,\tilde{x}_G) \in \cX^+ \cong V(\R) \times \cX_G^+$, define
\[
(V',G')(\tilde{x}_V,\tilde{x}_G) = \{(V'(\R) + g'\tilde{x}_V, g' \tilde{x}_G) : g' \in G'(\R)\}.
\]

\subsection{An auxiliary finiteness result}
Fix $\mathfrak{F}$ a fundamental set for $\bu \colon \cX^+ \rightarrow M$ such that $\bu|_{\mathfrak{F}}$ is definable in $\R_{\mathrm{an},\exp}$.

Consider the following definable set
\begin{align*}
\Upsilon = \{(\tilde{y}, \underline{g}, \underline{t}, v) \in (\bu^{-1}(Y) \cap \mathfrak{F}) \times \cG  \times \Omega \times V(\R) :  & ~ g G' g^{-1} \cdot g_V V' \subseteq g_V V', \\
 & ~ \tilde{y}(\S) \subseteq (v,1) \cdot N_P( g_V V' \rtimes g G' g^{-1}) \cdot (-v,1)\},
\end{align*}
where $\S$ is the Deligne torus as in Definition~\ref{DefnMixedShimuraDatum}, $\underline{g} = (g_V, g)$, $\underline{t} = (V',G')$, 
and $N_P(g_V V' \rtimes g G' g^{-1})$ is the normalizer of $g_V V' \rtimes g G' g^{-1}$ in $P_{\R}$.\footnote{The semi-direct product $g_V V' \rtimes g G' g^{-1}$ is well-defined by the assumption 
$g G' g^{-1} \cdot g_V V' \subseteq g_V V'$.}

\begin{lemma}\label{LemmaFinitenessPreliminaryAuxiliarySet}
We have 
\begin{enumerate}
\item[(i)] For any $(\tilde{y}, \underline{g}, \underline{t}, v) \in \Upsilon$, we have that $(\underline{g}\cdot \underline{t}) \tilde{y}$ is complex analytic (and hence complex algebraic since it is also semi-algebraic).
\item[(ii)] Any weakly special subset having non-empty intersection with $\bu^{-1}(Y) \cap \mathfrak{F}$ is of the form $(\underline{g} \cdot \underline{t})\tilde{y}$ for some $(\tilde{y}, \underline{g}, \underline{t}, v) \in \Upsilon$.
\end{enumerate}
\end{lemma}
\begin{proof}
\begin{enumerate}
\item[(i)] The set $(\underline{g}\cdot \underline{t}) \tilde{y}$ is
\[
\{(g_V V'(\R) + g g' g^{-1} \tilde{y}_V, g g' g^{-1} \tilde{y}_G) : g' \in G'(\R)\},
\]
where $\tilde{y} = (\tilde{y}_V,\tilde{y}_G)$ under $\cX^+ \cong V(\R) \times \cX_G^+$. Denote by $\tilde{y}' = (\tilde{y}_V - v, \tilde{y}_G)$. Then $(\underline{g}\cdot \underline{t}) \tilde{y}$ is the translate of the group orbit $(g_V V' \rtimes g G' g^{-1})(\R)\cdot \tilde{y}'$ by the constant section $\{v\}\times \cX_G^+$. The group orbit is complex analytic because $\tilde{y}'(\S) \subseteq N_P(g_V V' \rtimes g G' g^{-1})$. The constant section is complex analytic by the discussion in $\mathsection$\ref{SubsectionHodgeTheoreticX2ga} (between \eqref{EquationBettiFirstStep} and \eqref{EqTwoComplexStrOfX2g}). Hence $(\underline{g}\cdot \underline{t}) \tilde{y}$ is complex analytic.
\item[(ii)] Let $\tilde{X}$ be such a weakly special subset. Then there exist a connected mixed Shimura subdatum of Kuga type $(Q,\cY^+)$ of $(P,\cX^+)$, a normal subgroup $N$ of $Q$ whose reductive part is semi-simple, and a point $\tilde{y} \in \cY^+$ such that $\tilde{X} = N(\R)^+\tilde{y}$. We can take $\tilde{y} \in \bu^{-1}(Y) \cap \mathfrak{F}$. Then it suffices to take $\underline{g}\cdot \underline{t}$ to be $(V_{N,\R}, G_{N,\R}^{\mathrm{nc}})$, where $V_N = V \cap N$, $G_N = N/V_N$ and $G_{N,\R}^{\mathrm{nc}}$ is the product of the non-compact simple factors of $G_{N,\R}^+$. The element $v \in V(\R)$ appears because $N = (v,1)(V_N \rtimes G_N)(-v,1)$ for some $v \in V(\Q)$.\footnote{Here $V_N \rtimes G_N$ is defined to be compatible with the fixed Levi decomposition $P = V\rtimes G$. It may differ from $N$ as $\{0\}\rtimes G_N$ may not be contained in $N$.} \qedhere
\end{enumerate}
\end{proof}

From now on, all the dimensions are real dimensions. Define the following functions on $(\bu^{-1}(Y) \cap \mathfrak{F}) \times \cG  \times \Omega \times V(\R)$:
\[
\begin{array}{c}
d(\tilde{y}, \underline{g}, \underline{t}, v) = \dim_{\tilde{y}}\left( (\underline{g}\cdot \underline{t}) \tilde{y}  \right) = \dim_{\tilde{y}}\left( (g_V V', gG'g^{-1}) \tilde{y}  \right), \\
d_Y(\tilde{y}, \underline{g}, \underline{t}, v) = \dim_{\tilde{y}}\left( \bu^{-1}(Y) \cap \mathfrak{F} \cap (\underline{g}\cdot \underline{t}) \tilde{y}  \right) = \dim_{\tilde{y}}\left( \bu^{-1}(Y) \cap \mathfrak{F} \cap (g_V V', gG'g^{-1})  \tilde{y}  \right).
\end{array}
\]
Define 
\begin{align*}
\Xi_0 = \{ (\tilde{y}, \underline{g}, \underline{t}, v) \in \Upsilon : & ~ (\tilde{y}, \underline{g}_1, \underline{t}_1, v_1) \in \Upsilon,~ (\underline{g}\cdot \underline{t}) \tilde{y} \subsetneq (\underline{g}_1\cdot \underline{t}_1) \tilde{y} \\
& ~ \Rightarrow d(\tilde{y}, \underline{g}, \underline{t}, v) - d_Y(\tilde{y}, \underline{g}, \underline{t}, v) < d(\tilde{y}, \underline{g}_1, \underline{t}_1, v_1) - d_Y(\tilde{y}, \underline{g}_1, \underline{t}_1, v_1) \}.
\end{align*}
Finally define 
\[
\Xi = \{ (\tilde{y}, \underline{g}, \underline{t}, v) \in \Xi_0 : ~ (\tilde{y}, \underline{g}_1, \underline{t}_1, v_1) \in \Upsilon, ~ (\underline{g}\cdot \underline{t}) \tilde{y} \supsetneq (\underline{g}_1\cdot \underline{t}_1) \tilde{y}  \Rightarrow d_Y(\tilde{y}, \underline{g}, \underline{t}, v) > d_Y(\tilde{y}, \underline{g}_1, \underline{t}_1, v_1) \}.
\]
Then both $\Xi_0$ and $\Xi$ are definable.

\begin{lemma}\label{LemmagtFinite}
The set of pairs $\{(\underline{g}\cdot \underline{t}, v) : ~ (\tilde{y}, \underline{g}, \underline{t}, v) \in \Xi \}$ is finite.
\end{lemma}
\begin{proof}
We start by proving that $(\underline{g}\cdot \underline{t})\tilde{y}$ is weakly special for any $(\tilde{y}, \underline{g}, \underline{t}, v) \in \Xi$. Let $\tilde{Z}$ be the complex analytic irreducible component of $(\underline{g}\cdot \underline{t})\tilde{y} \cap \bu^{-1}(Y)$ passing through $\tilde{y}$ such that
\[
\dim \tilde{Z} = d_Y(\tilde{y}, \underline{g}, \underline{t}, v).
\]

Let $\tilde{Z}' \supseteq \tilde{Z}$ be such that $\tilde{Z}' \subseteq \bu^{-1}(Y)$ is complex analytic irreducible and 
$\delta_{\Zar}(\tilde{Z}') \le \delta_{\Zar}(\tilde{Z})$. We may and do assume that $\tilde{Z}'$ is Zariski optimal. Then $\tilde{Z}'$ is a complex analytic irreducible component of $(\tilde{Z}')^{\Zar} \cap \bu^{-1}(Y)$, and $(\tilde{Z}')^{\Zar} = (\underline{g}_1\cdot \underline{t}_1)\tilde{y}$ for some $(\tilde{y}, \underline{g}_1, \underline{t}_1, v_1) \in \Upsilon$ by Theorem~\ref{ThmASZariskiOptimal} and Lemma~\ref{LemmaFinitenessPreliminaryAuxiliarySet}.(ii). Here the $\tilde{y}$ can be taken as above.

Now we have $\tilde{Z} \subseteq (\underline{g}\cdot \underline{t})\tilde{y} \cap (\underline{g}_1\cdot \underline{t}_1)\tilde{y}$, and hence $\tilde{Z} \subseteq (\underline{g}_2\cdot \underline{t}_2)\tilde{y}$ for some $(\tilde{y}, \underline{g}_2, \underline{t}_2, v_2) \in \Upsilon$ with $(\underline{g}_2\cdot \underline{t}_2)\tilde{y} \subseteq (\underline{g}\cdot \underline{t})\tilde{y}$.\footnote{The point $\tilde{y}$ can be taken as before since $\tilde{y} \in \tilde{Z}$. Then $(\underline{g}_2\cdot \underline{t}_2, v_2)$ arises from the intersection of the two subgroups $(v,1)(g_V V' \rtimes g G' g^{-1})(-v,1)$ and $(v_1,1)(g_{V,1} V'_1 \rtimes g_1 G'_1 g_1^{-1})(-v_1,1)$ of $P$.} By definition of $\Xi$, we then have $(\underline{g}_2\cdot \underline{t}_2)\tilde{y} = (\underline{g}\cdot \underline{t})\tilde{y}$. Hence
\[
(\underline{g}\cdot \underline{t})\tilde{y} \subseteq (\underline{g}_1\cdot \underline{t}_1)\tilde{y}.
\]

On the other hand we have
\begin{align*}
d(\tilde{y}, \underline{g}_1, \underline{t}_1, v_1) - d_Y(\tilde{y}, \underline{g}_1, \underline{t}_1, v_1) & = \dim(\tilde{Z}')^{\Zar} - \dim_{\tilde{y}}((\tilde{Z}')^{\Zar} \cap \bu^{-1}(Y) \cap \mathfrak{F}) \\
& \le \delta_{\Zar}(\tilde{Z}') \le \delta_{\Zar}(\tilde{Z}) \\
& \le d(\tilde{y}, \underline{g}, \underline{t}, v) - d_Y(\tilde{y}, \underline{g}, \underline{t}, v).
\end{align*}
Hence by definition of $\Xi_0$, we have $(\underline{g}\cdot \underline{t})\tilde{y} = (\underline{g}_1\cdot \underline{t}_1)\tilde{y} = (\tilde{Z}')^{\Zar}$ is weakly special.

Thus by the definition of weakly special subvarieties, we have $\underline{g}\cdot \underline{t} = (V_{N,\R}, G_{N,\R}^{\mathrm{nc}})$ for some $\Q$-subgroup $N$ of $P$, where $V_N$ is the unipotent radical of $N$, and $G_{N,\R}^{\mathrm{nc}}$ is the almost direct product of the non-compact factors of $G_{N,\R}^+$. We also obtain some $v \in V(\Q)$; see the proof of Lemma~\ref{LemmaFinitenessPreliminaryAuxiliarySet}.(ii). Therefore the set $\{ (\underline{g}\cdot \underline{t}, v) : ~ (\tilde{y}, \underline{g}, \underline{t}, v) \in \Xi \}$ is countable.

On the other hand write $\Omega = \{\underline{t}_1,\ldots,\underline{t}_n\}$. Then we have $\Xi = \bigcup_{i=1}^n \Xi_i$, where $\Xi_i = \{(\tilde{y}, \underline{g}, \underline{t}, v) \in \Xi:~ \underline{t} = \underline{t}_i\}$. For each $i \in \{1,\ldots,n\}$, consider the map
\[
\Xi_i \rightarrow \left(\Sp_{2g}(\R)/\mathrm{Stab}_{\Sp_{2g}(\R)}(V'_i)\right) \times \left(G(\R)/N_{G(\R)}(G'_i) \right) \times V(\R) , \quad (\tilde{y}, \underline{g}, \underline{t}_i, v) \mapsto (g_V V'_i, g G'_i g^{-1}, v)
\]
where we write $\underline{t}_i = (V'_i, G'_i)$ and $\underline{g} = (g_V, g)$. This map is definable, and hence the image is definable. But its image is $\{(\underline{g}\cdot \underline{t}, v) : ~ (\tilde{y}, \underline{g}, \underline{t}, v) \in \Xi_i \}$. Thus
\[
\{(\underline{g}\cdot \underline{t}, v) : ~ (\tilde{y}, \underline{g}, \underline{t}, v) \in \Xi \} = \bigcup_{i=1}^n \{ (\underline{g}\cdot \underline{t}, v) : ~ (\tilde{y}, \underline{g}, \underline{t}, v) \in \Xi_i \}
\]
is definable. Hence this set is finite because it is countable and definable.
\end{proof}

\begin{lemma}\label{LemmaZariskiOptimalComeFromTriple}
Let $\tilde{Z} \subseteq \bu^{-1}(Y)$ be Zariski optimal such that $\tilde{Z} \cap \mathfrak{F} \not= \emptyset$. Then we have
\[
\tilde{Z}^{\Zar} = (\underline{g}\cdot \underline{t})\tilde{y}
\]
for some $(\tilde{y}, \underline{g}, \underline{t}, v) \in \Xi$.
\end{lemma}
\begin{proof} 
By Theorem~\ref{ThmASZariskiOptimal}, $\tilde{Z}^{\Zar}$ is weakly special. Hence part (ii) of Lemma~\ref{LemmaFinitenessPreliminaryAuxiliarySet} implies
\[
\tilde{Z}^{\Zar} = (\underline{g}\cdot \underline{t})\tilde{y}
\]
for some $(\tilde{y}, \underline{g}, \underline{t}, v) \in \Upsilon$. Moreover we may take $\tilde{y} \in \tilde{Z}$ such that $\dim \tilde{Z} = \dim_{\tilde{y}}(\tilde{Z}^{\Zar} \cap \bu^{-1}(Y) \cap \mathfrak{F})$. We wish to prove that $(\tilde{y}, \underline{g}, \underline{t}, v) \in \Xi$.

We start by proving $(\tilde{y}, \underline{g}, \underline{t}, v) \in \Xi_0$. Suppose not, then there exists $(\tilde{y}, \underline{g}_1, \underline{t}_1, v_1) \in \Upsilon$ such that $(\underline{g}\cdot \underline{t})\tilde{y} \subsetneq (\underline{g}_1\cdot \underline{t}_1)\tilde{y}$ and
\begin{equation}\label{EqDimensionInequalityForFiniteness}
d(\tilde{y}, \underline{g}, \underline{t}, v) - d_Y(\tilde{y}, \underline{g}, \underline{t}, v) \ge d(\tilde{y}, \underline{g}_1, \underline{t}_1, v_1) - d_Y(\tilde{y}, \underline{g}_1, \underline{t}_1, v_1).
\end{equation}
Let $\tilde{Z}'$ be a complex analytic irreducible component of $(\underline{g}_1\cdot \underline{t}_1)\tilde{y} \cap \bu^{-1}(Y)$ passing through $\tilde{y}$ such that $\dim \tilde{Z}' = d_Y(\tilde{y}, \underline{g}_1, \underline{t}_1, v_1)$. Then
\begin{align*}
\dim_{\tilde{y}}(\tilde{Z}' \cap \tilde{Z}^{\Zar}) & \ge \dim \tilde{Z}' + \dim \tilde{Z}^{\Zar} - d(\tilde{y}, \underline{g}_1, \underline{t}_1, v_1) \quad \text{ by the Dimension Intersection Inequality} \\
& \ge \dim \tilde{Z} \qquad \text{ by \eqref{EqDimensionInequalityForFiniteness}}.
\end{align*}
Thus $\tilde{Z}'$ contains a neighborhood of $\tilde{y}$ in $\tilde{Z}$, and hence $\tilde{Z} \subseteq \tilde{Z}'$. But $\tilde{Z}$ is Zariski optimal and $\delta_{\Zar}(\tilde{Z}') \le \delta_{\Zar}(\tilde{Z})$ by \eqref{EqDimensionInequalityForFiniteness}, so $\tilde{Z} = \tilde{Z}'$. But then we get the following contradiction to \eqref{EqDimensionInequalityForFiniteness}:
\[
d(\tilde{y}, \underline{g}_1, \underline{t}_1, v_1) - d_Y(\tilde{y}, \underline{g}_1, \underline{t}_1, v_1) \ge 2 \delta_{\Zar}(\tilde{Z}') = 2 \delta_{\Zar}(\tilde{Z}) = d(\tilde{y}, \underline{g}, \underline{t}, v) - d_Y(\tilde{y}, \underline{g}, \underline{t}, v).
\]
Here the first inequality follows from part (i) of Lemma~\ref{LemmaFinitenessPreliminaryAuxiliarySet}.

Hence $(\tilde{y}, \underline{g}, \underline{t}, v) \in \Xi_0$. Suppose this quadriple does not belong to $\Xi$, then there exists $(\tilde{y}, \underline{g}_1, \underline{t}_1, v_1) \in \Upsilon$ such that $(\underline{g}\cdot \underline{t})\tilde{y} \supsetneq (\underline{g}_1\cdot \underline{t}_1)\tilde{y}$ and $d_Y(\tilde{y}, \underline{g}_1, \underline{t}_1, v_1) = d_Y(\tilde{y}, \underline{g}, \underline{t}, v) = \dim \tilde{Z}$. But then
\[
\tilde{Z} \subseteq (\underline{g}_1\cdot \underline{t}_1)\tilde{y} \subsetneq (\underline{g}\cdot \underline{t})\tilde{y} = \tilde{Z}^{\Zar}.
\]
This is a contradiction to part (i) of Lemma~\ref{LemmaFinitenessPreliminaryAuxiliarySet}.
\end{proof}

\begin{prop}\label{PropAuxiliaryFinitenessProp}
There exists a finite set $\Sigma$ consisting of elements of the form $\left( (Q,\cY^+),N \right)$, where $(Q,\cY^+)$ is a connected mixed Shimura subdatum of $(P,\cX^+)$ and $N$ is a normal subgroup of $Q$ whose reductive part is semi-simple such that the following property holds. If $\tilde{Z}$ is a complex analytic irreducible subset in $\bu^{-1}(Y)$ which is Zariski optimal and such that $\tilde{Z} \cap \mathfrak{F} \not= \emptyset$, then there exists $\left( (Q,\cY^+),N \right) \in \Sigma$ such that $\tilde{Z}^{\Zar} = N(\R)^+\tilde{y}$ for some $\tilde{y} \in \cY^+$.
\end{prop}
\begin{proof}
By Theorem~\ref{ThmASZariskiOptimal}, we know that $\tilde{Z}$ is weakly special. Thus there exist a connected mixed Shimura subdatum of Kuga type $(Q,\cY^+)$ of $(P,\cX^+)$, a normal subgroup $N$ of $Q^{\der}$, and a point $\tilde{y} \in \cY^+$ such that $\tilde{Z} = N(\R)^+\tilde{y}$.

Let us prove that the $N$ arises from finitely many choices. It suffices to prove that $(V_N,G_N,v)$ arises from finitely many choices, where $V_N = V\cap N$, $G_N = N/V_N$ and $v \in V(\Q)$ such that $N = (v,1)(V_N \rtimes G_N)(-v,1)$. 
By Lemma~\ref{LemmaZariskiOptimalComeFromTriple} and Lemma~\ref{LemmagtFinite}, there are only finitely many choices for $(V_{N,\R}, G_{N,\R}^{\mathrm{nc}},v)$, where $G_{N,\R}^{\mathrm{nc}}$ is the almost direct product of the non-compact factors of $G_{N,\R}^+$. But then we can take $G_N$ to be the smallest connected $\Q$-subgroup of $G$ which contains $G_{N,\R}^{\mathrm{nc}}$. Hence we proved the finiteness of $N$.

Next for each $N$, there are only finitely many $(Q,\cY^+)$ such that $N \lhd Q$ by \cite[Lemma~12.1]{GaoTowards-the-And}. Hence we are done.
\end{proof}

\subsection{Proof of Theorem~\ref{ThmFinitenessAlaBogomolov}}
Let $Z$ be a closed irreducible subvariety of $Y$ which is weakly optimal. Let $\tilde{Z}$ be a complex analytic irreducible component of $\bu^{-1}(Z)$ such that $\tilde{Z} \cap \mathfrak{F} \not= \emptyset$. In view of Proposition~\ref{PropAuxiliaryFinitenessProp}, it suffices to prove that $\tilde{Z}$ is Zariski optimal in $\bu^{-1}(Y)$.

Let $\tilde{Z}' \supseteq \tilde{Z}$ be such that $\tilde{Z}' \subseteq \bu^{-1}(Y)$ is complex analytic irreducible and 
$\delta_{\Zar}(\tilde{Z}') \le \delta_{\Zar}(\tilde{Z})$. We may and do assume that $\tilde{Z}'$ is Zariski optimal. Then $\tilde{Z}'$ is a complex analytic irreducible component of $(\tilde{Z}')^{\Zar} \cap \bu^{-1}(Y)$, and $(\tilde{Z}')^{\Zar}$ is weakly special by Theorem~\ref{ThmASZariskiOptimal}.

On the other hand $\bu^{-1}(\bu(\tilde{Z}')^{\biZar}) \supseteq  (\tilde{Z}')^{\Zar}$ since $\bu(\tilde{Z}')^{\biZar}$ is bi-algebraic. So
\[
\bu(\tilde{Z}')^{\biZar} =  \bu \left( (\tilde{Z}')^{\Zar} \right).
\]
Hence we have
\begin{align*}
\delta_{\mathrm{ws}}(\bu(\tilde{Z}')^{\Zar}) & = \dim \bu(\tilde{Z}')^{\biZar} - \dim \bu(\tilde{Z}')^{\Zar}  = \dim (\tilde{Z}')^{\Zar} - \dim \bu(\tilde{Z}')^{\Zar} \\
& \le \dim (\tilde{Z}')^{\Zar} - \dim \tilde{Z}'  = \delta_{\Zar}(\tilde{Z}') \le \delta_{\Zar}(\tilde{Z})  = \dim \tilde{Z}^{\Zar} - \dim \tilde{Z} \\
& \le \dim Z^{\biZar} - \dim Z = \delta_{\mathrm{ws}}(Z).
\end{align*}
Since $\tilde{Z}' \subseteq \bu^{-1}(Y)$, we have $\bu(\tilde{Z}')^{\Zar} \subseteq Y$. Moreover recall our assumption that $Z$ is weakly optimal. So $Z = \bu(\tilde{Z}')^{\Zar}$. But then $\tilde{Z}' \subseteq \tilde{Z}$. So $\tilde{Z}$ is Zariski optimal. Hence we are done.

\section{A simple application to the Betti map}\label{SectionAppToBettiRank}
In this section, we present a simple application of the mixed Ax-Schanuel theorem for the universal abelian variety to the Betti map. Our goal is to show the idea, so we restrict ourselves to a simple case.

Our setting-up is as follows: let $S$ be an irreducible quasi-projective variety over $\C$ and let $\cA \rightarrow S$ be an abelian scheme of relative dimension $g$. By \cite[$\mathsection$2.1]{GenestierNgo}, $\cA \rightarrow S$ carries a polarization of type $D = \mathrm{diag}(d_1,\ldots,d_g)$ for some positive integers $d_1|d_2|\cdots|d_g$. Up to taking a finite covering of $S$, it induces a cartesian diagram
\[
\xymatrix{
\cA \ar[r]^-{\iota} \ar[d] & \mathfrak{A}_{g,D} \ar[d]^{\pi} \\
S \ar[r]^-{\iota_S} & \A_{g,D}.
}
\]
We furthermore make the following extra assumptions for simplicity:
\begin{itemize}
\item the geometric generic fiber of $\cA/S$ is a simple abelian variety.
\item $\iota_S$ is quasi-finite (so is $\iota$).
\end{itemize}

Because of the second bullet point, we may and do replace $\cA/S$ by $\iota(\cA)/\iota_S(S)$. Let  $\cA_{\mathfrak{H}_g^+}$ be the pullback of $\mathfrak{A}_{g,D} / \A_{g,D}$ under $\bu_G \colon \mathfrak{H}_g^+ \rightarrow \A_{g,D}$. 
Recall the real analytic diffeomorphism \eqref{EqiDelta}
\[
i_{\mathfrak{H}_g^+} \colon \R^{2g} \times \mathfrak{H}_g^+ \xrightarrow{\sim} \cX_{2g,\mathrm{a}}^+ = \lie(\cA_{\mathfrak{H}_g^+} / \mathfrak{H}_g^+).
\]
It induces then a real analytic diffeomorphism
\[
\bar{i}_{\mathfrak{H}_g^+} \colon \mathbb{T}^{2g} \times \mathfrak{H}_g^+ \xrightarrow{\sim} \cA_{\mathfrak{H}_g^+}.
\]
Hence we get the following map, which is called the \textbf{Betti map}
\[
b \colon \cA_{\mathfrak{H}_g^+} \xrightarrow{\sim} \mathbb{T}^{2g} \times \mathfrak{H}_g^+  \rightarrow \mathbb{T}^{2g}
\]
where the first map is $\bar{i}_{\mathfrak{H}_g^+}^{-1}$ and the last map is the projection.

Let $\tilde{S}$ be a complex analytic irreducible component of $\bu_G^{-1}(S)$, and let $\cA_{\tilde{S}}$ be the restriction of $\cA_{\mathfrak{H}_g^+}$ to $\tilde{S}$. By abuse of notation, we denote by
\[
b \colon \cA_{\tilde{S}} \rightarrow \mathbb{T}^{2g}
\]
the restriction of the Betti map.

\begin{thm}\label{ThmAppToBettiRank}
Let $\xi \colon S \rightarrow \cA$ be a multi-section. It induces a multi-section $\tilde{\xi}$ of $\cA_{\tilde{S}}/\tilde{S}$. Assume $\Z\xi$ is Zariski dense in $\cA$. If $\dim S \ge g$, then there exists $\tilde{s} \in \tilde{S}$ such that
\[
\mathrm{rank}(\mathrm{d}b|_{\tilde{\xi}(\tilde{s})}) = 2g.
\]
\end{thm}

\begin{proof}
Denote by $Y = \xi(S)$. It is an irreducible subvariety of $\mathfrak{A}_{g,D}$. Consider the diagram
\[
\xymatrix{
\cX_{2g,\mathrm{a}}^+  \ar[d]_{\bu} & \R^{2g} \times \mathfrak{H}_g^+ \ar[l]^-{i_{\mathfrak{H}_g^+}}_-{\sim} \\
\mathfrak{A}_{g,D}
}.
\]
Take a complex analytic irreducible component $\tilde{Y}$ of $\bu^{-1}(Y)$. By abuse of notation we identify $\cX_{2g,\mathrm{a}}^+$ and $\R^{2g} \times \mathfrak{H}_g^+$ and no longer write $i_{\mathfrak{H}_g^+}(\cdot)$.

Let $G = \mathrm{MT}(\tilde{S})$.  
Since $\Z\xi$ is Zariski dense in $\cA$, it is possible to take $\tilde{s} \in \tilde{S}$ with Mumford-Tate group $G$ and $a \in \R^{2g} = V_{2g}(\R)$ with some non-rational coordinate such that $(a,\tilde{s}) \in \tilde{Y}$. It suffices to prove that $\mathrm{rank}(\mathrm{d}b|_{(a,\tilde{s})}) = 2g$.

If $\mathrm{rank}(\mathrm{d}b|_{(a,\tilde{s})}) < 2g$, then by property (ii) below \eqref{EqiDelta} there exists a complex analytic variety $\tilde{C} \subseteq \tilde{S}$ of dimension $\ge \dim S - g + 1$ passing through $\tilde{s}$ such that $\{a\} \times \tilde{C} \subseteq \tilde{Y}$. Apply the weak Ax-Schanuel theorem (Theorem~\ref{ThmWASUnivAbVar}) to $\tilde{Z} = \{a\} \times \tilde{C}$, then we have
\begin{equation}\label{EqASForBetti1}
\dim (\{a\} \times \tilde{C})^{\Zar} + \dim \bu(\{a\} \times \tilde{C})^{\Zar} \ge \dim (\{a\} \times \tilde{C}) + \dim \bu(\{a\} \times \tilde{C})^{\biZar}.
\end{equation}
But $\{a\} \times \tilde{C}^{\Zar}$ is complex analytic (by property (ii) below \eqref{EqiDelta}) and real algebraic, and so is algebraic. Thus $(\{a\} \times \tilde{C})^{\Zar} = \{a\} \times \tilde{C}^{\Zar}$. On the other hand, the characterization of bi-algebraic subvarieties of $\mathfrak{A}_{g,D}$ and the assumption on $a$ imply that $\bu(\{a\} \times \tilde{C})^{\biZar} = \mathfrak{A}_{g,D}|_{\bu_G(\tilde{C})^{\biZar}}$; see Proposition~\ref{PropGeomDescriptionOfBiAlgebraic}.\footnote{Here we use the fact that the geometric generic fiber of $\cA/S$ is a simple abelian variety and that $\cA/S$ is non-isotrivial. In particular $\cA/S$ has no fixed part after finite base change.} So \eqref{EqASForBetti1} becomes
\begin{equation}\label{EqASForBettiTotalSpace}
\dim \tilde{C}^{\Zar} + \dim \bu(\{a\} \times \tilde{C})^{\Zar} \ge \dim S - g + 1 + \dim \mathfrak{A}_{g,D}|_{\bu_G(\tilde{C})^{\biZar}}.
\end{equation}
But $\dim \tilde{C}^{\Zar} \le \dim \tilde{C}^{\biZar} = \dim \bu_G(\tilde{C})^{\biZar}$. Moreover $\dim \bu(\{a\} \times \tilde{C})^{\Zar} = \dim \bu_G(\tilde{C})^{\Zar}$ since $\{a\} \times \tilde{C} \subseteq \tilde{Y}$ and $Y = \xi(S)$ is a multi-section of $\cA/S$. So we obtain from \eqref{EqASForBettiTotalSpace}
\[
\dim \bu_G(\tilde{C})^{\biZar} + \dim \bu_G(\tilde{C})^{\Zar} \ge \dim S - g + 1 + \dim \mathfrak{A}_{g,D}|_{\bu_G(\tilde{C})^{\biZar}}.
\]
Thus
\[
\dim \bu_G(\tilde{C})^{\Zar} \ge \dim S - g + 1 + (\dim \mathfrak{A}_{g,D}|_{\bu_G(\tilde{C})^{\biZar}} - \dim \bu_G(\tilde{C})^{\biZar}) = \dim S - g + 1 + g.
\]
Hence $\dim \bu_G(\tilde{C})^{\Zar}  \ge \dim S + 1$. But this is impossible as $\bu_G(\tilde{C})^{\Zar} \subseteq S$. So we get a contradiction.
\end{proof}

\end{document}